\pdfoutput=1
\documentclass{amsart}
\usepackage{amsmath, amsthm, amssymb}
\usepackage[margin=1.3in]{geometry}
\usepackage{tikz, graphicx, rotating}
	\usetikzlibrary{calc}

\usepackage{hyperref}
\usepackage{cleveref}
\usepackage{xfrac}
\usepackage{cite}

\newtheorem{theorem}{Theorem}[section]
\newtheorem{proposition}[theorem]{Proposition}
\newtheorem{lemma}[theorem]{Lemma}
\newtheorem{definition}[theorem]{Definition}

\newcommand{\R}{\mathbb R}

\newcommand{\cD}{\mathcal D}

\newcommand{\eps}{\varepsilon}

\newcommand{\dd}{\, \mathrm{d}}

\newcommand{\ww}{\langle w\rangle}
\newcommand{\vv}{\langle v\rangle}
\newcommand{\vvi}{\langle v_i\rangle}
\newcommand{\vvo}{\langle v_0\rangle}
\newcommand{\vvd}{\langle \cdot\rangle}
\newcommand{\vvp}{\langle v'\rangle}

\newcommand{\Qs}{Q_{\rm s}}
\newcommand{\Qns}{Q_{\rm ns}}

\newcommand{\AC}{\begin{figure}[t]
	\newcommand{\RAD}{1} %Radius of circle
	\newcommand{\SCAL}{.6*\RAD} % scaling factor related to perpendicular vectors
	\begin{tikzpicture}
	
		\coordinate (V) at (0,4*\RAD);
		\coordinate (VP) at (.6, -.35);
		
		\path let \p1=($(V) - (VP)$) in coordinate (VPerp) at (\y1,- \x1);
		\draw[->] (0,0) -- (V)
			node[above left] {\scriptsize $v$};
		
		\draw[|-|, gray] (.1,0) -- (.1,\RAD - .025);% node[midway, left]{$\sfrac{|v|}{4}$};
		
		\fill (0,0) circle (.02) node[below left] {\scriptsize $0$};
		\draw[->, red] (0,0) -- (VP) node[below left] {\scriptsize$v'$};
		\draw[->, dashed, red] (VP) -- (V) node[midway, right] {\scriptsize$v-v'$};
		\draw[blue] ($(V) - \SCAL*(VPerp)$) -- ($(V) + \SCAL*(VPerp)$)
			node[below]{\scriptsize $v + (v-v')^\perp$};%(-4.42/2, -.615/2 +4) -- (4.42/2,.615/2 + 4) 
  
		\draw (0,0) circle (\RAD);
	\end{tikzpicture}
	\begin{tikzpicture}
	
		\coordinate (V) at (0,4*\RAD);
		\coordinate (VP) at (.6, -.35);
		\path let \p1=($(V) - (VP)$) in coordinate (VPerp) at (\y1,- \x1);
		
		\coordinate (U1) at ($\SCAL*4.4*(1, .2582)$);
		\path let \p1=(U1) in coordinate (U2) at (\x1,- \y1);
		
		\draw[->] (0,0) -- (V);
		\draw[red,->] ($(V) + (.3,.25)$) arc (5:355:.3 and .15);
		
		\draw[|-|, gray] (.1,0) -- (.1,\RAD - .05) node[midway, right]{$\sfrac{|v|}{4}$};
		
		\fill (0,0) circle (.02) node[below left] {\scriptsize $0$};
		\fill[red!20] (V) -- ($(V) + (U1)$) -- ($(V)+(U2)$) -- cycle;
		\draw[red] (V) -- ($(V) + (U1)$);
		\draw[red] (V) -- ($(V)+(U2)$);
		\node[red] at ($(V) + .5*(U1) + .5*(U2)$) {$\mathcal{AC}$};

		\fill[red!20] (V) -- ($(V) - (U1)$) -- ($(V) - (U2)$) -- cycle;
		\draw[red] (V) -- ($(V) - (U1)$);
		\draw[red] (V) -- ($(V) - (U2)$);
		
		\draw[blue] ($(V) - \SCAL*(VPerp)$) -- ($(V) + \SCAL*(VPerp)$);
		\draw (0,0) circle (\RAD);
	\end{tikzpicture}
\caption{On the left, we draw a typical plane $v + (v-v')^\perp$ in which $v+w$ lives.  On the right, we draw our larger set $\mathcal{AC}$ (the shaded region) that contains all such planes.  We point out that these regions exist in three dimensions, so that, e.g., the shaded region $\mathcal{AC}$ is really the rotation of the figure above around the $z$-axis.}\label{f.AC}
\end{figure}}

\tikzset{offset/.style={to path={%
    -- ($(\tikztostart)!#1cm!(\tikztotarget)$)}},
         offset/.default=1}

\newcommand{\DI}{\begin{figure}[t]
	\newcommand{\RAD}{1} %Radius of B_{|v|/4}
	\newcommand{\SCAL}{.6*\RAD} % scaling factor related to perpendicular vectors
	\newcommand{\RI}{.135*\RAD} %Radius of B_{r_i}(v_i)
	\begin{tikzpicture}
	
		\coordinate (V) at (0,4*\RAD); %vector v
		\coordinate (VI) at (3.2*\RAD, 4.5*\RAD); %vector v_i
		
		\path let \p1=($(V) - (VI)$) in coordinate (VPerp) at (\y1,-\x1);
		\coordinate (VIT) at ($(VI) + .3*\RI*(VPerp)$); % location where tangent line going through v touches top of B_{r_i}(v_i)
		\coordinate (VIB) at ($(VI) - .3*\RI*(VPerp)$); % same as above but instead tangent line on bottom
		
		\path let \p2=($(V) - (VIT)$) in coordinate (VTPerp) at (\y2,-\x2);
		\path let \p3=($(V) - (VIB)$) in coordinate (VBPerp) at (\y3,-\x3);
		\node[below,red] at (VIB) {\tiny $\displaystyle B_{r_i}(v_i)$};

		\draw[red, fill=red!20] (VI) circle (\RI); %ADD LABLE!
		\draw[blue, dashed] (V) -- (VIT);
		\draw[blue, dashed] (V) -- (VIB);
		
		\draw[opacity=0] (V) --($2*(V) - (VIB)$);%$.2*(V)$):
		\draw[blue, dashed] (V) -- ($(V)-1.75*\RAD*(VTPerp)$);
		\draw[blue, dashed] (V) -- ($(V)-1.75*\RAD*(VBPerp)$);
		
		\node[blue] at (.25,.25) {\tiny $\cD_i$};

		\begin{scope}
			\clip (0,0) circle (\RAD);
			\fill[blue!30] (V) -- ($(V)-1.75*\RAD*(VTPerp)$) -- ($(V)-1.75*\RAD*(VBPerp)$) -- cycle;
		\end{scope}
		
		\draw[dotted] (V) -- ($(V)-1.65*\RAD*(VPerp)$)
			node[below]{\tiny$P_i$};

		\draw[->] (0,0) -- (V)
			node[above left] {\scriptsize $v$};
		
		\fill (0,0) circle (.02) node[below left] {\scriptsize $0$};
		
		\draw (0,0) circle (\RAD);
	\end{tikzpicture}
	\qquad\qquad\qquad
	\begin{tikzpicture}
	
		\coordinate (V) at (0,4*\RAD); %vector v
		\coordinate (VI) at (3.2*\RAD, 4.5*\RAD); %vector v_i
		
		\path let \p1=($(V) - (VI)$) in coordinate (VPerp) at (\y1,-\x1);
		\coordinate (VIT) at ($(VI) + .3*\RI*(VPerp)$); % location where tangent line going through v touches top of B_{r_i}(v_i)
		\coordinate (VIB) at ($(VI) - .3*\RI*(VPerp)$); % same as above but instead tangent line on bottom

		\path let \p2=($(V) - (VIT)$) in coordinate (VTPerp) at (\y2,-\x2);
		\path let \p3=($(V) - (VIB)$) in coordinate (VBPerp) at (\y3,-\x3);

		\draw[opacity=0] (VI) circle (\RI); %ADD LABLE!
		\draw[opacity=0] (V) -- (VIT);
		\draw[opacity=0] (V) -- (VIB);
		\draw[opacity=0] (V) -- ($(V)-1.75*\RAD*(VTPerp)$);
		\draw[opacity=0] (V) -- ($(V)-1.75*\RAD*(VBPerp)$);
		
		\node[blue] at (.95,.95) {\tiny $E\cD_i$};

		\begin{turn}{-9.55}
			\clip (0,0) circle (\RAD);
			\fill[blue!30] (V) -- ($(V)-1.75*\RAD*(VTPerp)$) -- ($(V)-1.75*\RAD*(VBPerp)$) -- cycle;
			\draw[dotted] (V) -- ($(V)-1.75*\RAD*(VPerp)$);
		\end{turn}
		
		\draw[opacity=0] (V) -- ($(V)-1.65*\RAD*(VPerp)$)
			node[below]{\tiny$P_i$};

		\draw[opacity=0,->] (0,0) -- (V)
			node[above left] {\scriptsize $v$};
		
		\fill (0,0) circle (.02) node[below left] {\scriptsize $0$};
		
		\draw (0,0) circle (\RAD);
	\end{tikzpicture}
	\caption{On the left, the blue shaded region represents $\cD_i$.  It is clear from this picture that this region is nearly a disk, with each ``slice'' in the $y$-$z$ plane ($y$ being the vertical direction and $z$ being the direction extending into the page) being a circle.  On the right, we rotate $\cD_i$ by $E$ as in the proof of~\eqref{e.c100501}.} \label{f.cD}
\end{figure}}

\newcommand{\parameters}{\begin{figure}[t]
	\centering
	\begin{tikzpicture}
		\fill[olive!10] (0,0) -- (0,2) -- (-6,2) -- (-6,0) -- cycle;
		
		\draw[red] (0,0) -- (-4,2) node[midway,below,sloped]{\tiny $\gamma+2s = 0$};
		\fill[red!10] (0,0) -- (-4,2) -- (0,2) -- cycle;

		\draw[blue] (0,1) -- (-1,2) node[midway, below, sloped]{\tiny $\gamma+s=\sfrac12$} ;
		\fill[blue!10] (0,1) -- (-1,2) -- (0,2) -- cycle;

		\draw[gray, thick, ->] (0,0) -- (0,2.5) 
			node[right,gray] {$s$};
		\draw[dotted, very thick] (-6,0) -- (0,0) --(0,2) -- (-6,2) 
			node[midway, above, gray]{\tiny Landau equation ($s=1$)} -- (-6,0);
		\draw[gray, thick, ->] (0,0) -- (-7,0) node[below left, gray]{$\gamma$};
		\draw[very thick] (-.125,2) -- (.125,2);
		\draw[very thick] (-6,-.125)--(-6,.125);
		\node at (.25,2) {\tiny $1$};
		\node at (.25,0) {\tiny $0$};
		\node at (0,-.3)  {\tiny $0$};
		\node at (-6.1,-.3) {\tiny $-3$};

	\end{tikzpicture}
	\caption{Above are the three parameter regimes that determine which hydrodynamic conditions in~\eqref{e.conditional2} are required in \Cref{t:decay2}.  The olive area is $\gamma + 2s <0$, the red is $\gamma + 2s > 0$ but $\gamma + s < \sfrac12$, and the blue region is $\gamma + s > \sfrac12$. As discussed below, the Landau equation arises when $s\to1$, which, in this parameter space, is the dotted line running horizontally along the top.}\label{f.parameters}
\end{figure}}

\newcommand{\pk}{p_0}

\newcommand{\vertiii}[1]{{\left\vert\kern-0.25ex\left\vert\kern-0.25ex\left\vert #1 
    \right\vert\kern-0.25ex\right\vert\kern-0.25ex\right\vert}}
    
\newcommand{\KOO}[3]{\vertiii{#1}_{#2,#3}}
\newcommand{\KO}[1]{\KOO{#1}{q_0}{\pk}}

\usepackage{esint,dsfont}
\newcommand{\1}{\mathds{1}}

\newcommand{\be}{\begin{equation}}
\newcommand{\ee}{\end{equation}}

\Crefname{definition}{Definition}{Definitions}

\numberwithin{equation}{section}
\numberwithin{theorem}{section}
\title{Decay estimates and continuation for the non-cutoff Boltzmann equation}

\author{Christopher Henderson}
\address{Department of Mathematics, University of Arizona, Tucson, AZ 85721}
\email{ckhenderson@math.arizona.edu}

\author{Stanley Snelson}
\address{Department of Mathematical Sciences, Florida Institute of Technology, Melbourne, FL 32901}
\email{ssnelson@fit.edu}

\author{Andrei Tarfulea}
\address{Department of Mathematics, Louisiana State University, Baton Rouge, LA 70803}
\email{tarfulea@lsu.edu}

\thanks{CH was supported by NSF grant DMS-2204615 and acknowledges support of the Institut Henri Poincaré (UAR 839 CNRS-Sorbonne Université) and LabEx CARMIN (ANR-10-LABX-59-01). SS was supported by NSF grant DMS-2213407. AT was supported by NSF grant DMS-2108209.}

\begin{document}

\maketitle

\begin{abstract}
We consider the non-cutoff Boltzmann equation in the spatially inhomogeneous, soft potentials regime, and establish decay estimates for large velocity. In particular, we prove that pointwise algebraically decaying upper bounds in the velocity variable are propagated forward in time whenever the solution has finite weighted $L^\infty_{t,x} L^p_v$-norms for certain $p$.  The main novelty is that these estimates hold for {\em any} decay exponent above $\max\{2,3 + \gamma\} +2s$, where $\gamma$ and $s$ are standard physical parameters such that $\gamma \in (-3,0)$ and $s\in (0,1)$. Our results are useful even for solutions with mild decay.

As an application, we combine our decay estimates with recent short-time existence results to derive a continuation criterion for large-data solutions.  Compared to past results, this extends the range of allowable parameters and weakens the requirements on smoothness and decay in velocity of solutions.
\end{abstract}

\section{Introduction}
We consider the Boltzmann equation, a fundamental kinetic model from statistical physics. The unknown function $f(t,x,v)\geq 0$ represents the density of particles at time $t\geq 0$, location $x\in \R^3$, and velocity $v\in \R^3$. The equation reads
\begin{equation}\label{e.boltzmann}
	(\partial_t +v\cdot\nabla_x) f = Q(f,f),
\end{equation}
where the bilinear collision operator is defined for functions $f, g:\R^3\to \R$ by
\begin{equation}\label{e.Q}
 Q(f,g) := \int_{\R^3} \int_{\mathbb S^2} B(v-v_*,\sigma) [f(v_*')g(v') - f(v_*) g(v)] \dd \sigma \dd v_*.
 \end{equation}
For any deviation angle $\sigma \in \mathbb S^2$, the pre- and post-collisional velocities are related by the formulas
\be 
	v' = \frac{v+v_*} 2 + \frac{|v-v_*|} 2 \sigma,
	\qquad
	v_*' = \frac{v+v_*} 2 - \frac{|v-v_*|} 2 \sigma,
\ee
and the {\it non-cutoff} collision kernel $B$ is defined by
\be \label{e.B}
	B(v-v_*,\sigma) = b(\cos\theta) |v-v_*|^\gamma,
	\qquad \cos\theta = \sigma \cdot \frac{v-v_*}{|v-v_*|},
\ee
for some $\gamma \in (-3,1]$, where $b$ satisfies the bounds
\be 
c_b \theta^{-2-2s} \leq b(\cos\theta) \leq \frac 1 {c_b} \theta^{-2-2s}, 
\ee
for some $s\in (0,1)$ and $c_b>0$. In this paper, we assume 
\be
	\gamma \in (-3,0)
	\quad\text{ and }\quad s\in(0,1).
\ee
This is the full range of possible $s$, while $\gamma$ is restricted to the ``soft potentials'' case.

Our goal is to investigate the decay of solutions for large velocities. This is a classical topic in the study of the Boltzmann equation.  Its interest is due to its importance in establishing quantitative control on the nonlocal collision operator.  This is needed, for example, in order to prove regularity estimates or establish conservation of energy. 
Because the kinetic kernel $|v-v_*|^\gamma$ in \eqref{e.B} decays slowly, some decay of $f$ is needed even to evaluate $Q(f,f)$ pointwise. 

In the soft potentials regime ($\gamma<0$) considered here, it is well-understood that decay estimates are not self-generating, i.e. the decay of $f$ is essentially limited by the decay of the initial data. Therefore, the best result we can hope for is of the following form: if the initial data satisfies an estimate like $f_{\rm in}(x,v) \leq K \vv^{-q}$, then the solution also satisfies $f(t,x,v) \leq K'\vv^{-q}$ for some $K'$ depending on $K$ and some (hopefully mild) norms of $f$. When $\gamma > 0$, one expects a stronger result: that polynomial decay estimates hold regardless of the initial data. We do not consider this case in the current article, because the proof strategy would necessarily be different.

All the prior decay estimates for the regime we consider (non-cutoff, spatially inhomogeneous, far from equilibrium) apply only for large polynomial decay exponents \cite{imbert2018decay, cameron2020boltzmann}. In the $\gamma< 0$ case, such decay estimates require a strong assumption on the decay of the initial data, because of the issue mentioned in the previous paragraph. By contrast, the decay estimates in the present article hold for any exponent larger than $\max\{3+ \gamma,2\}+2s$ and therefore apply to solutions whose initial data might decay quite slowly.  Our decay estimates also seem to be the first that cover the parameter regime $\gamma+2s<0$ for inhomogeneous, large-data solutions.

We are also motivated by the {\it continuation of large-data solutions}. While the open question of unconditional large-data global existence for \eqref{e.boltzmann} is likely out of reach with current techniques, in recent years there has been interesting partial progress in the form of continuation criteria, most notably \cite{imbert2020smooth} which established that, in the case $\gamma+2s\in [0,2]$, solutions with infinite order of decay in $v$ can be continued for as long as the mass, energy, and entropy densities are bounded above, and the mass density is bounded below, uniformly in $t$ and $x$. Later, the present authors showed in \cite{HST2020lowerbounds} that one can discard the lower mass bound and upper entropy bound from the criterion of \cite{imbert2020smooth}, i.e. that solutions can be continued as long as 
\be \label{e.mod-soft}
\int_{\R^3} f(t,x,v)\dd v \leq M_0 \quad \text{ and } \quad \int_{\R^3} |v|^2 f(t,x,v) \dd v \leq E_0,
\ee
uniformly in $t$ and $x$. Again, this criterion holds only for the regime $\gamma+2s\in [0,2]$ and requires $f_{\rm in}(x,\cdot)$ to have infinite order of decay.
Until now, no similar continuation criterion has been available for the $\gamma+2s<0$ case nor has the restriction on the decay in $v$ been relaxed. 
%In \Cref{c:cont} below \CH{(MAYBE WE SHOULD UPGRADE THIS TO A THEOREM?)}, we combine our decay estimates with recent short-time existence results to derive a continuation criterion for this remaining case. This new criterion is more restrictive than \eqref{e.mod-soft}, which is to be expected because the singularity $|v-v_*|^\gamma$ in \eqref{e.B}, a key source of difficulty, is more severe when $\gamma$ is more negative.  \CH{On the other hand, it requires only that the initial data satisfies $f_{\rm in}(x,v) \leq K \vv^{-q}$ {\em some} $q> 3+2s$.}

\subsection{Main results}

To state our results, let us introduce the following notation for weighted $L^p$-norms on $\R^3$: for $p \in [1,\infty]$ and $k \in \R$, define
\be
	\|f\|_{L^p_k(\R^3)} := \| \vv^k f\|_{L^p(\R^3)},
\ee
where $\vv = \sqrt{1+|v|^2}$.  When the $x$ variable is included, we abuse notation and write
\be
	\|f\|_{L^p_k(\R^6)} := \| \vv^k f\|_{L^p(\R^6)},
\ee
Similarly for when the $t$ and $x$ variables are included.

Our first main result says that decay estimates for the initial data are propagated forward to positive times, as long as some weaker norms of $f$ remain finite.

\begin{theorem}\label{t:decay1}
Let $f\geq 0$ be a classical solution to the Boltzmann equation~\eqref{e.boltzmann} on $[0,T]\times\R^6$, with the initial data $f_{\rm in}$ lying in $L^\infty_q(\R^6)$ for some $q> \max\{3 + \gamma,2\} + 2s$.  If $f$ satisfies the bounds
\be\label{e.hydro}
\begin{split}
\int_{\R^3} \vv^{(1+2\gamma+2s+\delta)_+} f(t,x,v) \dd v &\leq K_0,\\
\int_{\R^3} f^{\frac{3}{3+\gamma}+\eta} (t,x,v) \dd v &\leq G_0,\\
\int_{\R^3} \vv^{p(1+2\gamma+2s+\delta)_+} f^p(t,x,v) \dd v &\leq P_0, \quad \text{(only necessary if $\gamma + 2s < 0$)}, %\\
%\int_{\R^3} |v|^{\ell} f(t,x,v) \dd v &\leq L_0, \quad \text{(only necessary if $\gamma + s \geq 1/2$),}
%\|f(t,x,\cdot)\|_{L^1_{(1+2\gamma+2s+\delta)_+}(\R^3)}
\end{split}
\ee
uniformly in $(t,x) \in [0,T]\times\R^3$, for some $p>3/(3+\gamma+2s)$, some small $\delta,\eta>0$ and some constants $K_0$, $G_0$, and $P_0$, then one has
\be
	\|f\|_{L^\infty_q([0,T]\times \R^6)} \leq K,
\ee
for some $K>0$ depending only on $T$, $\|f_{\rm in}\|_{L^\infty_q(\R^6)}$, $K_0$, $G_0$,  $\gamma$, $s$, $c_b$, and (if $\gamma+2s< 0$) $P_0$ and $p$. %, and (if $\gamma+s \geq 1/2$) $L_0$ and $\ell$.
\end{theorem}

%We should note that Theorem \ref{t:decay1} requires no lower bounds for $f$ or the initial data. 
Theorem \ref{t:decay1} can be understood as a regularization, as it bounds a stronger norm $\|f(t,x,\cdot)\|_{L^\infty_q(\R^3)}$ in terms of the weaker norms in \eqref{e.hydro} plus the initial data. Regularization estimates for the non-cutoff Boltzmann equation almost always require some quantitative lower bounds for $f$. Heuristically, if $f$ is small, there are few particles to collide with, so the diffusive effects of the collision term will be insignificant. Therefore, it is notable that Theorem \ref{t:decay1} requires no lower bounds for $f$ or the initial data.

By contrast, our next main result weakens the conditional assumption on $f$ for positive times, at the cost of assuming suitable lower bounds for the initial data. We say the initial data $f_{\rm in}$ is {\it well-distributed with parameters $c_m, r, R>0$} if for all $x\in \R^3$, there exists an $(x_m,v_m) \in B_R(x,0)$ with 
\be\label{e.well}
	f_{\rm in} \geq c_m 1_{B_r(x_m,v_m)}.
\ee
Roughly, this condition says that for every physical location $x$ there are some particles of ``not too large'' velocity ``nearby.'' 
% at every location $x$ in physical space. \st{Morally, this condition means that every location in the $x$ domain is too far from some mass of $f$ that lies within relatively bounded velocities.}}
In the special case of spatially periodic initial data, \eqref{e.well} would follow from uniform positivity in some small ball in velocity. Our second result is as follows:

\begin{theorem}\label{t:decay2}
Let $f\geq 0$ be a classical solution to the Boltzmann equation~\eqref{e.boltzmann} on $[0,T]\times\R^6$, with the initial data $f_{\rm in}$ lying in $L^\infty_q(\R^6)$ for some $q> 3 + \gamma + 2s$. In addition, assume that the initial data is well-distributed with parameters $c_m, r, R$.   Fix $p>3/(3 + \gamma + 2s)$, $\ell > 2\gamma+2s+1$, and $k = \max\{1-3(\gamma+2s)/(2s), 1+2\gamma+2s+\delta\}$ for some small $\delta>0$.

If $f$ satisfies the bounds, for all $(t,x) \in [0,T]\times \R^3$,
\begin{equation}\label{e.conditional2}
\begin{split}
	\int_{\R^3} f(t,x,v) \dd v &\leq M_0,\\
	\int_{\R^3} |v|^2 f(t,x,v) \dd v &\leq E_0,\\
	\int_{\R^3} |v|^{k} f^{p} (t,x,v) \dd v &\leq P_0, \quad \text{(only necessary if $\gamma+2s\leq 0$),}\\
	\int_{\R^3} |v|^{\ell} f(t,x,v) \dd v &\leq L_0, \quad \text{(only necessary if $\gamma + s \geq 1/2$),}
\end{split}
\end{equation}
for some constants $M_0$, $E_0$, $P_0$, and $L_0$, then one has
\be
	\|f\|_{L^\infty_q([0,T]\times \R^6)} \leq K
\ee
for some $K>0$ depending only on $T$, $\|f_{\rm in}\|_{L^\infty_q(\R^6)}$, $c_m$,  $r$, $R$, $M_0$, $E_0$, $\gamma$, $s$, $c_b$, as well as (if $\gamma+2s\leq 0$) $P_0$ and $p$, and (if $\gamma + s \geq 1/2$) $L_0$ and $\ell$.
\end{theorem}

\parameters

Note that the lower bound in $p$ in Theorem \ref{t:decay2} is much less restrictive than the norm $L^{3/(3+\gamma)+\eta}_v$ in Theorem \ref{t:decay1}, which converges to $L^\infty_v$ as $\gamma \searrow -3$. 

Next, we have a continuation criterion, which follows from combining Theorem \ref{t:decay2} with the recent short-time existence result of \cite[Theorem~1.1]{HST2022irregular}:

\begin{theorem}\label{c:cont}
Let $f\geq 0$ be a solution to the Boltzmann equation~\eqref{e.boltzmann} on $[0,T]\times\R^6$, with the initial data $f_{\rm in}$ lying in $L^\infty_q(\R^6)$ for some $q>3+2s$. Assume that $f_{\rm in}$ is well-distributed with parameters $c_m$, $r$, $R$. 

If $f$ cannot be extended to a solution on $[0,T+\eps)\times\R^6$ for any $\eps>0$, then one of the inequalities in \eqref{e.conditional2} must degenerate as $t\nearrow T$.
\end{theorem}

%\stan{We could also prove a continuation criterion that does not require well-distributed initial data, but would need the conditional bounds from Theorem \ref{t:decay1} including the $L^\infty_{t,x}L^{3/(3+\gamma)+\eps}_v$ norm for all $\gamma, s$. I'm not sure if this is worth stating?}

As mentioned above, this appears to be the first continuation criterion for large-data solutions in the case $\gamma+2s<0$.  Additionally, \Cref{c:cont} is a new result because it applies under weaker assumptions than the continuation criterion of \cite{HST2020lowerbounds}, which applied to solutions that are smooth, periodic in $x$, and rapidly decaying in $v$.
If $\gamma+s \geq 1/2$ (see \Cref{f.parameters}), then \Cref{c:cont} contains an extra decay assumption (the last inequality in \eqref{e.conditional2}) that is not present in \cite{HST2020lowerbounds}. We believe this assumption is technical. %Indeed, sending $s\to 1$, we obtain the Landau equation (see, e.g.,~\cite[Section 1.7]{villani2002review}), whose continuation criteria \cite[Theorem~1.3]{HST2020landau} requires no such extra condition.

Let us comment on the lower bound $p>3/(3+\gamma+2s)$ in Theorem \ref{t:decay2} and \Cref{c:cont}, which appears in the case $\gamma + 2s < 0$. 
If we send $s\to 1$ in this condition, we recover exactly the borderline exponent in the weakest known continuation criterion for the inhomogeneous Landau equation, which corresponds to the limit of the Boltzmann equation \eqref{e.boltzmann} as $s\to 1$. In more detail, the main result of \cite{snelson2023landau} says that, when $\gamma < -2$, solutions to Landau can be continued past time $T$ provided
\be
\int_{\R^3} f(t,x,v)\dd v \leq M_0, \quad \int_{\R^3} |v|^\ell f(t,x,v) \dd v \leq L_0, \quad \text{ and } \|f(t,x,\cdot)\|_{L^p_v(\R^3)} \leq P_0,
\ee
uniformly in $(t,x) \in [0,T]\times\R^3$, for some $p>3/(\gamma+5)$, some arbitrarily small $\ell>0$, and some constants $M_0, L_0, P_0$. This condition on $p$ also appears in continuation and regularity criteria for the space homogeneous Landau equation (see e.g. \cite{gualdani2017landau, ABDL, golding2023local}) and seems to be the sharpest condition available with current techniques. Based on this comparison with the Landau equation, we believe that removing the higher $L^p_v$ bound from our condition \eqref{e.conditional2}, or even improving the lower bound $p>3/(3+\gamma+2s)$, would be a difficult task. By contrast, removing the weight $|v|^k$ from the $L^p$ norm in \eqref{e.conditional2} is most likely within reach, although still nontrivial.

\subsection{Proof ideas}

The key step in our proof is an upper bound for the ``singular'' part of the collision operator. To explain what this means, let us recall the Carleman decomposition, which allows one to write $Q(f,g)$ as the sum of two terms:
\begin{equation}\label{e.QsQns}
Q(f,g) = Q_{\rm s}(f,g) + Q_{\rm ns}(f,g).
\end{equation}
The singular term $Q_{\rm s}$ is an integro-differential operator of order $2s$, given by
\begin{equation}\label{e.Qs}
\begin{split}
Q_{\rm s}(f,g) &= \int_{\R^3} K_f(v,v') [g(v') - g(v)]\dd v',\\
K_f(v,v') &= |v-v'|^{-3-2s} \int_{w\in (v'-v)^\perp} f(v+w) |w|^{\gamma+2s+1} A(v,v',w)\dd w,
\end{split}
\end{equation}
where the function $A$ is uniformly positive and bounded, i.e. $A(v,v',w) \approx 1$, with implied constants depending only on $\gamma$, $s$, and $c_b$. 
The non-singular term $Q_{\rm s}$ can be written as
\begin{equation}\label{e.Qns}
Q_{\rm ns}(f,g) = C g(v) \int_{\R^3} f(v+z) |z|^\gamma \dd z,
\end{equation}
where $C>0$ is a constant depending only on $\gamma$, $s$, and $c_b$. We refer to \cite{alexandre2000noncutoff, alexandre2000entropy, silvestre2016boltzmann} for proofs of the formulas \eqref{e.Qs} and \eqref{e.Qns}.

Our key estimate, Lemma \ref{l:Qs-pointwise-base}, is an upper bound for $\Qs(f,g)$ with $g(v) = \vv^{-q}$. This bound takes the following form:  for any $q\geq \min\{3+\gamma+2s, 2+2s\}$,
\be\label{e.Qs-est}
\Qs(f,\langle \cdot\rangle^{-q})  \lesssim \vv^{-q} 
\begin{cases}
\|f\|_{L^1_{q-2+\gamma}(\R^3)}, &\gamma+2s\geq 0,\\
\|f\|_{L^1_{q-2+\gamma}(\R^3)} + \|f\|_{L^p_{q-2+\gamma}(\R^3)}, &\gamma+2s<0,
\end{cases}
\ee
where $p>3/(3+\gamma+2s)$. %This estimate, Lemma , may be of independent interest.

Let us briefly discuss the difficulties of proving \eqref{e.Qs-est}. 
%Proving Lemma \ref{l:Qs-pointwise-base} is the most difficult part of this article. To indicate some of the obstructions, let us recall that $\Qs(f,g)$ can be written 
%\be
%\Qs(f,g) = \int_{\R^3} K_f(v,v') [g(v') - g(v)] \dd v',
%\ee
%where $K_f(v,v')$ is a kernel that is singular of order $2s$ and depends nonlocally on $f$. (See \eqref{e.Qs}.) %Estimating the tail of the $v'$ integral, i.e. $\{v' : |v'|\geq |v|/4\}$ in \eqref{e.Qs-est} follows in a relatively standard way from known upper bounds satisfied by $K_f(v,v')$ and a Taylor expansion of $g$ when $v' \sim v$. The region near $v'=0$ 
%
%
%
Looking at the form of $\Qs(f,g)$ in \eqref{e.Qs}, we can see that when $|v'|$ is large, one can take advantage of the smallness of $g(v') = \langle v'\rangle^{-q}$ to bound the tail of the integral. In the region where $|v'|$ is small, the term $\langle v'\rangle^{-q}$ is approximately 1, so to get a good upper bound, one must instead use the smallness of the kernel $K_f(v,v')$, which is much more subtle.  Heuristically, if $|v|$ is much larger than $|v'|$, $v - v'$ is almost parallel to $v$.  Since $f\in L^1_{q-2+\gamma}(\R^3)$, most of its mass is located near the origin.  This implies that the hyperplane $v + (v-v')^\perp$ defining $K_f(v,v')$ (recall \eqref{e.Qs}) is located far from the origin and, thus, only intersects regions where $f$ is small,  making $K_f(v,v')$ small.
%then since most of the mass of $f\in L^1_{q-2+\gamma}(\R^3)$ is located near the origin,  the vector $v'-v$ is likely to be almost parallel to $v$ in regions where $f(v')$ is large. 
%This implies the hyperplane $v + (v-v')^\perp$ in the integral defining $K_f(v,v')$ (see \eqref{e.Qs}) does not usually come close to the region near $0$ where $f$ has more mass.  This can be exploited to show $K_f(v,v')$ is small.   
See Figure \ref{f.AC} in Section \ref{s.main_lemma}. 
%In some sense, the difficulty of Lemma \ref{l:Qs-pointwise-base} is to find the weakest norm of $f$ that can control all three of the following parts of the integral: (i) the singularity at $v'\sim v$, (ii) the region with $v'\sim 0$, and (iii) the unbounded tail.  We note that (ii) is the most challenging region to bound for large $|v|$, because we want to obtain a factor of $\vv^{-q}$ in the final estimate, but the  function $g(v') = \langle v'\rangle^{-q}$ is approximately 1 and gives no benefit. Therefore, we must use the smallness of $K_f(v,v')$ when $v$ is much larger than $v'$, which 
%heuristically follows from the fact that  ``most'' of the mass of $f\in L^1_{q-2+\gamma}(\R^3)$ is near zero, so $v'-v$ is ``usually'' almost parallel to $v$, so the hyperplane $v + (v-v')^\perp$ defining $K_f(v,v')$ (see \eqref{e.Qs}) does not ``usually'' come close to the region near $0$ where $f$ is large. 
A difficulty in quantifying this heuristic is that the integral defining $K_f$ is on a hyperplane and, thus, cannot be controlled using the $L^1_{q-2+\gamma}$-norm of $f$.  We can, however, make this precise in an averaged sense by applying a delicate localization method that is well-suited to the nonlocal dependence of $K_f(v,v')$ on $f$.

  It is nontrivial to make this heuristic picture precise and quantitative using only the $L^1_{q-2+\gamma}$-norm of $f$ to control an integral on a lower dimensional hyperplane. To do this, we apply a delicate localization method that is well-suited to the nonlocal dependence of $K_f(v,v')$ on $f$.

Armed with \eqref{e.Qs-est}, we prove decay estimates for $f$ using a barrier argument with barriers of the form $e^{\beta t} \vv^{-q}$. 
The crucial fact is that this argument gives a small gain of decay, i.e. we can bound an $L^\infty_q$ norm of $f$ in terms of a slightly weaker $L^\infty_{q'}$ norm of $f$, and the initial data. Iterating this gain, we eventually obtain a bound for any polynomially-weighted $L^\infty$ norm of $f$ that is finite at $t=0$. This bootstrapping procedure is partially inspired by our work in \cite{HST2022irregular}.

Once Theorem \ref{t:decay1} is available, we can combine it with the $L^\infty$-estimate of \cite{silvestre2016boltzmann} and the lower bounds of \cite{HST2020lowerbounds} to quickly obtain Theorem \ref{t:decay2}. Finally, \Cref{c:cont} follows from combining  Theorem \ref{t:decay2} with the short-time existence result \cite[Theorem~1.1]{HST2022irregular} with $L^\infty_q$ initial data.

\subsection{Comment about linear vs. nonlinear barrier arguments}

The barrier arguments that we use are {\it linear} in the following sense: if $f$ is a solution and $g$ is an upper barrier, then at a crossing point of $f$ and $g$, one has $Q(f,f) \leq Q(f,g)$ by the positivity properties of the collision operator. We then use a bilinear estimate of the form \eqref{e.Qs-est} and the finiteness of the appropriate norms of $f$, to bound this term from above and derive a contradiction.  This should be contrasted with {\it nonlinear barrier arguments} as seen in \cite{silvestre2016boltzmann, imbert2018decay}, where one additionally uses the fact that $f\leq g$ throughout all of $\R^3_v$ at the first crossing time, to estimate the kernel $K_f(v,v')$. In other words, the inequality $f\leq g$ is used to bound both copies of $f$ appearing in $Q(f,f)$, which results in a nonlinear dependence on the barrier $g$ in the final upper bound.

Our intuition is the following: nonlinear barrier arguments usually apply only for high rates of decay (Gaussian or high-degree polynomial) and give estimates that do not degenerate for large time, while linear barrier arguments hold for both slow and fast decay, but often result in estimates that degenerate for large time. The same dichotomy appears in the study of the Landau equation: see the linear barrier arguments applied in \cite{cameron2017landau, HST2020landau, S2018hardpotentials} and the nonlinear barrier arguments in \cite{snelson2023landau}.

\subsection{Related work}

As mentioned above, velocity decay estimates for the Boltzmann equation are a classical topic. Here, we give a brief and non-exhaustive overview of some important results.

Early work focused on the space homogeneous, cutoff Boltzmann equation. For polynomial $L^1$ moments in this setting, see \cite{ikenberry1956kinetic, elmroth1983boltzmann, desvillettes1993moments, wennberg1996povzner, lu1999boltzmann} and the references therein. %Many of these works made use of Povzner identities (see \cite{povzner}).
For exponentially-weighted $L^1$ estimates, see \cite{bobylev1997moment, mischler2006cooling, alonso2013exponential}. Regarding pointwise decay estimates, we refer to \cite{carleman1933boltzmann, arkeryd1983boltzmann} for results on polynomial decay, and \cite{gamba2009upper} for pointwise exponential decay, still in the homogeneous, cutoff setting. Later, \cite{gamba2019exponential} extended pointwise exponential decay estimates to the non-cutoff, homogeneous equation.

For the spatially inhomogeneous Boltzmann equation, see \cite{gualdani2017htheorem} for polynomial decay estimates with a hard-spheres collision kernel. The result of Imbert-Mouhot-Silvestre \cite{imbert2018decay} is more similar to ours because it addressed the inhomogeneous, non-cutoff Boltzmann equation. Working in the case $\gamma+2s\in [0,2]$, the authors of \cite{imbert2018decay} established polynomial decay estimates $f\lesssim \vv^{-q}$ for sufficiently large, non-explicit $q$, and these estimates do not degenerate for large time (conditional on some assumptions for the solution $f$). By contrast, our decay estimates hold for any $q$ larger than a small constant, but they grow with time.  The result of \cite{cameron2020boltzmann} extended the analysis of \cite{imbert2018decay} to the case $\gamma+2s>2$. The current article provides the only estimates of this type for the non-cutoff equation in the regime $\gamma+2s< 0$.

\subsection{Notation}

In addition to the notation $\vv$ and $L^p_k$ introduced above, we use $A \lesssim B$ to denote $A \leq C B$ where $C$ is a constant allowed to depend on the global parameters $\gamma$, $s$, and $c_b$, and other constants when specified. %, constants $q$, $c_m$, $r$, $R$, $p$, $\ell$, $k$, $K_0$, $G_0$, $P_0$, $M_0$, $E_0$,  and $L_0$.}  
We write $A \approx B$ to mean $A\lesssim B$ and $B \lesssim A$.

By classical solution to~\eqref{e.boltzmann}, we mean $f \in C^{2s + \alpha}_\ell$ for some $\alpha>0$ and $f$ satisfies~\eqref{e.boltzmann} pointwise.  See \cite[Section~3]{imbert2020smooth} for more details on the kinetic H\"older space $C^{2s +\alpha}_\ell$.  We assume this regularity throughout the paper even when it is not mentioned.

\subsection{Organization of the paper}

In Section \ref{s:prelim}, we discuss some useful estimates for the collision operator. Section \ref{s:decay} establishes our first decay estimate, Theorem \ref{t:decay1}, taking for granted the key Lemma \ref{l:Qs-pointwise-base}. Section \ref{s:better-decay} establishes the second decay estimate Theorem \ref{t:decay2}, as well as the continuation criterion of \Cref{c:cont}, and Section \ref{s.main_lemma} contains the proof of Lemma \ref{l:Qs-pointwise-base}. Finally, Appendix \ref{s:conv} contains a technical estimate for convolutions.

\section{Preliminaries}\label{s:prelim}

%\subsection{Carleman decomposition}\label{s:carleman}

%To analyze the bilinear collision operator $Q(f,g)$, we use the Carleman decomposition, which is by now a standard tool in the study of the non-cutoff Boltzmann equation. 

%\subsection{Estimates for the collision operator}

%\begin{lemma}{\cite[Lemma 2.3]{imbert2020lowerbounds}}\label{l:Qs-pointwise}
%If $g:\R^3\to \R$ is bounded and $C^2$, then
%\be |Q_{\rm s}(f,g)| \leq C\left( \int_{\R^3} |w|^{\gamma+2s} f(t,x,v-w) \dd w\right) \|g\|_{L^\infty(\R^3)}^{1-s} \|D_v^2 g\|_{L^\infty(\R^3)}^s.\ee
%\end{lemma}

In this section, we recall some known estimates for the collision operator. They are all stated in the context of the Carleman decomposition \eqref{e.QsQns}.

First, we have an upper bound for the average of the kernel $K_f(v,v')$ on annuli centered at $v$:

\begin{lemma}{\cite[Lemma 4.3]{silvestre2016boltzmann}}\label{l:K-upper-bound}
Let $K_f(v,v')$ be defined as in \eqref{e.Qs}. For any $r>0$, 
\be
	\int_{B_{2r}\setminus B_r} K_f(v,v+w) \dd w
		\lesssim \left( \int_{\R^3} f(v+w)|w|^{\gamma+2s} \dd w\right) r^{-2s},
\ee
for an implied constant depending only on $\gamma$, $s$, and $c_b$. 
\end{lemma}
The following equivalent statements can be proven by writing the integral over $B_r$ as a sum of integrals over $B_{r2^{-n}}\setminus B_{r2^{-n-1}}$ for $n=0,1,\ldots$ and applying Lemma \ref{l:K-upper-bound} for each $n$ (see, e.g., \cite[Lemma~2.2]{henderson2021existence}).
\begin{lemma}\label{l:K-upper-bound-2}
For any $r>0$, 
\begin{align}
	\int_{B_{r}} K_f(v,v+w) |w|^2 \dd w
		&\lesssim \left( \int_{\R^3} f(v+w)|w|^{\gamma+2s} \dd w\right) r^{2-2s},\\
	\int_{\R^3\setminus B_r} K_f(v,v+w) \dd w	
	&\lesssim \left( \int_{\R^3} f(v+w)|w|^{\gamma+2s} \dd w\right) r^{-2s},
\end{align}
for an implied constant depending only on $\gamma$, $s$, and $c_b$. 
\end{lemma}

The following estimate for $Q_{\rm s}(f,g)$ is useful on domains that are bounded in velocity:

\begin{lemma}{\cite[Lemma 2.3]{imbert2020lowerbounds}}\label{l:Qs-interp}
For any $f,g$ such that the right-hand side is finite, there holds
\be
|Q(f,g)| \lesssim \left( \int_{\R^3} |w|^{\gamma+2s} f(t,x,v-w) \dd w \right) \|g\|_{L^\infty(\R^3)}^{1-s} \|D_v^2 g\|_{L^\infty(\R^3)},
\ee
for an implied constant depending only on $\gamma$, $s$, and $c_b$. 
\end{lemma}

The following estimate in terms of weighted $L^\infty$ norms will be used in our gain-of-decay argument:

\begin{lemma}{\cite[Lemma 2.16]{HST2022irregular}}\label{l:Q-polynomial}
For any $q_1> 3 + \gamma + 2s$, let $f\in L^\infty_{q_1}(\R^3)$ be a nonnegative function, and choose $q_2\in [q_1, q_1-\gamma]$. (Recall that $\gamma < 0$.) Then there holds
\be 
	Q(f,\langle \cdot\rangle^{-q_2})(v)
		\lesssim \|f\|_{L^\infty_{q_1}(\R^3)}\vv^{-q_2}.
\ee
The implied constant depends on $q_1$, $q_2$, $\gamma$, $s$, and $c_b$. 
\end{lemma}

Finally, we have an upper bound for $\Qns$. We omit the proof, which follows from \eqref{e.Qns} and standard convolution estimates.

\begin{lemma}\label{l:Qns}
For any $f:\R^3\to \R$ such that the right-hand side is finite, and any $p>3/(3+\gamma)$, the estimate
\be
	Q_{\rm ns}(f,g)(v)
		\lesssim g(v)\left(\|f\|_{L^1(\R^3)} +  \|f\|_{L^p(\R^3)}\right)
\ee
holds, with the implied constant depending on $p$, $\gamma$, $s$, and $c_b$. 
\end{lemma}

\section{Polynomial decay of solutions}\label{s:decay}

\subsection{The $q_0, p_0$ norm}

We begin by defining a norm $\KO{\cdot}$ that plays an important role in our estimates of the collision operator. This norm will be convenient for estimating certain convolutions involving a kernal like $|v|^{\gamma+2s}$ (see Lemma \ref{l:convolution} below). The definition of this norm depends on $\gamma \in (-3,0)$ and $s\in (0,1)$, as well as the value of the exponents $q_0$ and $\pk$. For functions $f$ defined on $\R^3$, define
%The $K$ here indicates that $\pk$ will be used to estimate the kernel $K_f$. 
%$\delta \in (0, 3+\gamma+2s)$ and define
%\be
%	p_\delta = \frac{3 - \delta}{3 - \delta + \gamma+2s}.
%\ee
%with $p_\delta'$ its conjugate exponent.  Then set
%\CH{$\delta$ is undefined here.  also I think we only need $k_0 = \gamma + 2s$ in the $\gamma+2s\geq 0$ case.}
%\begin{equation}
%k_0 = \left\lbrace
%\begin{split}
%&\gamma + 2s && \quad\quad \text{ if } \gamma < -1 \\
%&2\gamma + 2s + 1+ \delta && \quad\quad \text{ if } \gamma \geq -1
%\end{split}\right.
%\end{equation}
%%and
%\begin{equation}\label{eq:K0_definition}
%	\vertiii{f}_{\gamma,s}??? = \KO{f} := \left\lbrace
%	\begin{split}
%		& \| f \|_{L^1_{1+2\gamma+2s+\delta}(\R^3)}
%			&& \quad\quad \text{ if }
%				\gamma + 2s \geq 0, &&&\gamma + 1 \geq 0, 
%		\\& \| f \|_{L^1_{\gamma+2s}(\R^3)}
%			 && \quad\quad \text{ if }
%			 	\gamma + 2s \geq 0, &&&\gamma + 1 < 0, 
%		\\& \| f \|_{L^1(\R^3)} + \| f \|_{L^{\pk}_{1+\gamma}(\R^3)}
%			&& \quad\quad \text{ if } 
%				 \gamma +2s < 0, &&&\gamma+1 \geq 0,		
%		\\& \| f \|_{L^1(\R^3)} + \| f \|_{L^{\pk}(\R^3)}
%			&& \quad\quad \text{ if }
%				\gamma+2s < 0, &&&\gamma + 1 < 0.
%\end{split}\right.
%\end{equation}
%Then, for any $\delta \geq 0$, set \CH{PROBABLY DELTA SHOULD BE BAKED INTO THIS NOTATION}
\begin{equation}\label{eq:K0_definition}
	\KO{f}
	=
		%= \KO{f} := \left\lbrace
	\left\{\begin{split}
		& \| f \|_{L^1_{q_0 - 2 + \gamma}(\R^3)}
			&& \qquad \text{ if }
				\gamma + 2s \geq 0,
		\\
		& \| f \|_{L^1_{q_0 - 2 + \gamma}(\R^3)} + \| f \|_{L^{\pk}_{q_0 - 2 + \gamma}(\R^3)}
			&&\qquad \text{ if } \gamma +2s < 0.
		 %&&&\gamma + 1 \geq 0, 
%		\\& \| f \|_{L^1_{\gamma+2s}(\R^3)}
%			 && \qquad \text{ if }
%			 	\gamma + 2s \geq 0, &&&\gamma + 1 < 0, 
%		\\& \|f\|_{L^1_{q_0 - 2+\gamma}(\R^3)} + \| f \|_{L^{\pk}_{??? ?? 1+\gamma}(\R^3)}
%			&& \qquad \text{ if } 
%				 \gamma +2s < 0, &&&\gamma+1 \geq 0,		
%		\\& \| f \|_{L^1(\R^3)} + \| f \|_{L^{\pk}(\R^3)}
%			&& \qquad \text{ if }
%				\gamma+2s < 0, &&&\gamma + 1 < 0.
\end{split}\right.
\end{equation}
We choose a $\pk$ satisfying
\be\label{e.p0def}
	\pk > \frac{3}{3 + \gamma + 2s},
\ee
and this value for $\pk$ will remain fixed for the remainder of the paper. The assumptions on $q_0$ will be specified in the statement of each lemma or theorem below. By allowing $q_0$ to vary, we can state sharper versions of our lemmas, for the sake of potential future applications.

The norm $\KO{\cdot}$ is, of course, independent of $\pk$ in the moderately soft potentials case $\gamma+2s\geq 0$; however, we include $\pk$ in the subscript of $\KO{\cdot}$ in order to avoid using different notation for the two cases.  
%Let us note that the $q_0$ below will be taken to be
%\be\label{e.c100301}
%	q_0 = \begin{cases}
%		3 + \gamma + 2s + \delta
%			\qquad&\text{ whenever } \gamma \geq -1,\\
%		2+2s
%			\qquad&\text{ whenever } \gamma < -1.
%		\end{cases}
%\ee 
%Here $\delta$ is any (preferably small) positive constant. The lemma below does not require this precise value, so we prove it in generality.

When $f$ is, additionally, a function of $t$ and $x$, we abuse notation and write $\KO{f}$ to also refer to
\be
	\sup_{t,x} \KO{f(t,x,\cdot)}.
\ee

\subsection{Upper bounds for $Q_{\rm s}$}

We now state the key estimate on which the proof of our theorem depends.  It is a bound for $Q_{\rm s}(f,\cdot)$ applied to negative powers of $\vv$.

\begin{lemma}\label{l:Qs-pointwise-base}
Suppose that $q_0 \geq \min\{3 + \gamma + 2s, 2 + 2s\}$ and $f$ is a nonnegative measurable function. Then
\begin{equation}\label{eq:Qs-pointwise-base}
	Q_{\rm s}(f, \langle \cdot \rangle^{-q_0})(v)
		\lesssim \KO{f} \vv^{-q_0}.
\end{equation}
The implied constant depends on $\gamma$, $s$, $q_0$, and (when $\gamma+2s<0$) $p_0$.
\end{lemma}
Despite its importance to our argument, we postpone its rather lengthy proof until we show how to leverage it to obtain \Cref{t:decay1}.  The proof is contained in \Cref{s:Qs-pointwise-base-proof}.

The next estimate relates to $\Qs$ applied to {\em positive} powers of $\vv$. % (cf. \Cref{l:Qs-pointwise-base}, which relates to {\em negative} powers).  
 It is almost certainly not sharp, and its role is purely technical and related to the choice of barrier.  %It is, however, necessary to our arguments due to the choice of barrier. 
Its proof is contained in \Cref{s.v-alpha}. 

%here relates to $\Qs(f,\cdot)$ applied to {\em positive} powers of $\vv$ (cf. \Cref{l:Qs-pointwise-base}, which relates to {\em negative} powers).  It is required due to the specific choice of our barrier.  

\begin{lemma}\label{l:v-alpha}
For any $\alpha\in (0,2s)$, $q_0 \geq \max\{3+ \gamma + 2s, 2+2s\}$,  and nonnegative measurable function $f$, we have %and $f\geq 0$, the estimate 
\be
\begin{split}
	Q_{\rm s}(f,\langle \cdot\rangle^\alpha)(v) 
& 	\lesssim \vv^\alpha \int_{\R^3} f(v+w) |w|^{\gamma+2s} \dd w \\
		&\lesssim  \vv^{\alpha - \theta_{\gamma,s}} \KO{f},
\end{split}
\ee
whenever the right-hand side is finite.  Here $\theta_{\gamma,s} = 2s - (\gamma+2s)_+> 0$.  The implied constants depend on $\alpha$, $\gamma$, $s$, and $c_b$.
\end{lemma}

%\stan{Maybe put the proof here, since it's short?}

\subsection{Propagation of weighted $L^\infty$-norms}

We now leverage the above lemmas in an iteration scheme to yield our decay estimates.  This next lemma is to understand $L^\infty_{q_0}$ estimates when $q_0$ is ``close'' to the starting base case value. 

\begin{lemma}\label{l:base-estimate}
Let $\eta$ be any positive number, $q_0\geq \max\{3 + \gamma + 2s, 2+2s\}$, and
\be
	f\in L^\infty([0,T]\times \R^6)
\ee
be a nonnegative classical solution of the Boltzmann equation \eqref{e.boltzmann} on $[0,T]\times \R^6$ with the initial data satisfying
\be
 f_{\rm in} \in L^\infty_{q_0}(\R^6) .
\ee
Assume further that %$f\in L^\infty([0,T]\times\R^6)$ and that 
$\KO{f}$ as defined in \eqref{eq:K0_definition} is finite. 
Then $f \in L^\infty_{q_0}([0,T]\times\R^6)$, and 
\be
	\|f(t) \|_{L^\infty_{q_0}(\R^6)}
		\leq \|f_{\rm in}\|_{L^\infty_{q_0}(\R^6)} \exp(C t (\KO{f} + \|f\|_{L^\infty_{t,x}L^{3/(3+\gamma) + \eta}_v([0,T]\times\R^6)})),
\ee
where the constant $C>0$ depends only on $q_0$, $p_0$, $\eta$, $\gamma$, $s$, and $c_b$. 
\end{lemma}

\begin{proof}
For any small $\eps>0$, let
\be
	N = (1+\eps) \|f_{\rm in}\|_{L^\infty_{q_0}(\R^6)}.
\ee
Fix any $\alpha \in (0,2s)$, and let $\beta>0$ be a constant to be determined.  Define the barrier
\be
	g(t,x,v) = N e^{\beta t} \left( \vv^{-q_0} + (\eps\vv)^\alpha + \eps e^{\eps^2\langle x \rangle}\right).
\ee
Define the first crossing time as
\be\label{e.c092701}
	t_0 := \sup \{ \tilde t_0 \in [0,T] : f(t) < g(t) \quad\text{ for all } t \in [0,\tilde t_0]\}.
\ee
If $t_0 = T$, we are finished.  Hence, we assume that $t_0 < T$.

We now show that $t_0 > 0$ and that there is a crossing point.  Because $f\in L^\infty$ and $g$ grows as $(x,v)$ tend to infinity, there is $R_\eps$ such that $f < g$ on $B_{R_\eps}^c$.  It follows from continuity that there is $z_0 := (t_0,x_0, v_0)$ with
\be
	g(z_0) - f(z_0) = \min_{[0,t_0] \times \overline B_{R_\eps}} (g - f) = 0.
\ee
The last equality follows from the definition~\eqref{e.c092701} of $t_0$.  Since $N > \|f_{\rm in}\|_{L^\infty_{q_0}(\R^6)}$, the inequality $f < g$ holds at $t=0$.  We deduce that $t_0 > 0$.

By elementary calculus, we find
%Because $f$ is bounded and $g$ approaches $\infty$ as $|x| + |v|$ approaches $\infty$, continuity of $(\partial_t + v\cdot \nabla_x )f$ implies that $t_0 > 0$ and thathe first crossing point $z_0 := (t_0,x_0,v_0)$ must occur with $t_0$ strictly positive. At this first crossing point, we have
\begin{equation}\label{e.crossing}
\begin{split}
	&f(z_0) = g(z_0),
	\qquad (\partial_t + v_0\cdot\nabla_x) f(z_0) \geq (\partial_t + v_0 \nabla_x)g(z_0),
	\qquad\text{ and}
	\\& f(t_0,x_0,v) \leq g(t_0,x_0,v) \quad \text{ for all } v\in \R^3,%,\\
%(\partial_t + v\cdot\nabla_x) f(z_0) &\geq \partial_t g(z_0),%\\
% \nabla_x f(z_0) &= \nabla_x g(z_0),
\end{split}
\end{equation}
so that
\begin{equation}\label{e.contra}
\begin{split}
	(\partial_t + v_0 \cdot \nabla_x) g(z_0)
		\leq (\partial_t + v_0 \cdot \nabla_x) f(z_0) 
		= Q(f,f)(z_0).
\end{split}
\end{equation}
Our goal is now to bound the left-hand side from below and the right-hand side from above to obtain a contradiction with~\eqref{e.contra}.

A direct computation shows that the left-hand side of~\eqref{e.contra} is equal to
\be\label{e.c092702}
N e^{\beta t_0}\left[\beta  \left(\langle v_0 \rangle^{-q_0} + (\eps\langle v_0\rangle)^\alpha + \eps e^{\eps^2\langle x_0\rangle} \right) + \eps^3 v_0\cdot\frac {x_0} {\langle x_0\rangle}e^{\eps^2\langle x_0\rangle} \right]. 
\ee
To bound the last term in~\eqref{e.c092702} from below, note that, since $f(z_0) = g(z_0)$, one has
\be
	(\eps\langle v_0\rangle)^\alpha \leq \|f\|_{L^\infty},
	\quad\text{ which implies that }\quad
	|v_0|\leq \langle v_0\rangle \leq \frac{1}{\eps}\|f\|_{L^\infty}^{\sfrac1\alpha}.
\ee	
Therefore,
\be
	\left|\eps^3 v_0 \cdot \frac {x_0}{\langle x_0\rangle} e^{\eps^2\langle x_0\rangle} \right|
		\leq \eps^2 \|f\|_{L^\infty}^{\sfrac1\alpha} e^{\eps^2\langle x_0 \rangle},
\ee
and we conclude that
\begin{equation}\label{e:g-lower}
	\partial_t g(z_0) + v_0\cdot \nabla_x g(z_0)
		\geq N e^{\beta t_0} \left[ \beta  \left(\langle v_0 \rangle^{-q_0} + (\eps\langle v_0\rangle)^\alpha + \eps e^{\eps^2\langle x_0\rangle} \right) - \eps^2 \|f\|_{L^\infty}^{\sfrac1\alpha} e^{\eps^2\langle x_0\rangle}\right].
\end{equation}

Next, we bound the right-hand side of~\eqref{e.contra} from above.  Writing $Q(f,f)$ in the Carleman form \eqref{e.QsQns} and using the properties in \eqref{e.crossing}, we see that
\be
\begin{split}
	Q(f,f)(z_0)
		&= Q_{\rm s}(f,f)(z_0) + Q_{\rm ns}(f,f)(z_0)
%\\
%& = \int_{\R^3} K_f(t_0,x_0,v_0,v')[f(t_0,x_0,v') - f(t_0,x_0,v_0)] \dd v' + c_s f(z_0) [f(t_0,x_0,\cdot)\ast |\cdot|^\gamma](v_0)\\
%&
		\leq Q_{\rm s}(f,g)(z_0) + Q_{\rm ns}(f,g)(z_0).
\end{split}
\ee
Using Lemma \ref{l:Qs-pointwise-base} and Lemma \ref{l:v-alpha}, 
\be\label{e.c092703}
	Q_{\rm s}(f,g)(z_0)
		\leq CN e^{\beta t_0}\left( \KO{f}\langle v_0 \rangle^{-q_0}
			+ \eps^\alpha \KO{f} \langle v_0\rangle^{\alpha-\theta_{\gamma,s}}\right), %\leq CN e^{\beta t_0}\left( \KO{f}\langle v_0 \rangle^{-q} + (\eps \langle v_0\rangle)^{\alpha}\right).
\ee
since $Q_{\rm s}(f,e^{\eps^2\langle x\rangle}) = 0$ (as with any function constant in $v$). 
For $Q_{\rm ns}$, Lemma \ref{l:Qns} implies
 \be
 \begin{split}
 &Q_{\rm ns}(f,g)(z_0) \leq C N \left(\|f\|_{L^\infty_{t,x}L^{\frac{3}{3+\gamma}+\eta}_v}%([0,T]\times\R^6)}
 	 + \|f(t_0,x_0,\cdot)\|_{L^1}%(\R^3)}
	 \right) e^{\beta t_0} \left(\langle v_0\rangle^{-q_0} + (\eps\langle v_0 \rangle)^\alpha + \eps e^{\eps^2\langle x_0\rangle}\right).
 \end{split}
 \ee
Note that $\|f(t_0,x_0,\cdot)\|_{L^1(\R^3)}$ is bounded by $\KO{f}$ in all cases.  Combining this estimate with~\eqref{e.c092703},~\eqref{e.contra}, and~\eqref{e:g-lower} and simplifying, we obtain
%\begin{equation}\label{e:Qupper}
%\begin{split}
%Q(f,f)(z_0) &\leq C N\KO{f}  e^{\beta t_0} \langle v_0\rangle^{-q_0}\\
%&\quad + C N (\|f\|_{L^\infty_{t,x}L^{3/(3+\gamma)+\eta}_v([0,T]\times\R^6)} + \|f(t_0,x_0,\cdot)\|_{L^1(\R^3)}) e^{\beta t_0} \langle v_0\rangle^{-q_0}
%\end{split}
%\end{equation}
% Combining \eqref{e:Qupper} with \eqref{e.contra} and simplifying, we obtain 
\begin{equation}\label{e:beta}
\begin{split}
	\beta&\left(\langle v_0\rangle^{-q_0} +\eps^\alpha\langle v_0\rangle^\alpha + \eps e^{\eps^2 \langle x_0\rangle}\right)
		\\&
		\leq C \langle v_0\rangle^{-q_0} \left(\KO{f} + \|f\|_{L^\infty_{t,x}L^{\frac{3}{3+\gamma}+\eta}_v}\right)
			+ C\eps^\alpha \KO{f} \langle v_0 \rangle^{\alpha} + \eps^2\|f\|_{L^\infty}^{\sfrac1\alpha} e^{\eps^2\langle x_0 \rangle},
\end{split}
\end{equation}
and we obtain a contradiction with the choice
\be
	\beta 
		= 2 C\left(\KO{f} + \|f\|_{L^\infty_{t,x}L^{\frac{3}{3+\gamma}+\eta}_v}\right)
			+ \eps\|f\|_{L^\infty}^{\sfrac1\alpha}.
\ee
As a result, it must hold that $t_0 = T$; that is $f\leq g$ on $[0,T]\times \R^6$.  This implies that, for all $\eps>0$,
\be
\begin{split}
	f(t,x,v) &< (1+\epsilon) \| f _{\rm in} \|_{L^\infty_{q_0}} \exp\left(t\left(2 C\left(\KO{f} + \|f\|_{L^\infty_{t,x}L^{\frac{3}{3+\gamma}+\eta}_v}\right) + \eps\|f\|_{L^\infty}^{\sfrac1\alpha}\right)\right)\\
&\quad \times \left( \vv^{-q_0} + (\eps\vv)^\alpha + \eps e^{\eps\langle x \rangle}\right).
\end{split}
\ee
The proof is finished after sending $\eps\to 0$.
\end{proof}

The previous lemma gives propagation of $L^\infty_{q_0}$ norms with $q_0$ fairly small. The next lemma uses another barrier argument to improve this decay by a small amount.

\begin{lemma}\label{l:gain}
Fix any $q_1$ and $q_2$ satisfying
\be
%	q_1 \geq 2 + 2s,
%	\qquad
	q_1 > 3 + \gamma + 2s
	\qquad\text{and}\qquad
	q_2 \in [q_1,q_1+|\gamma|].
\ee
Let $f\in L^\infty_{q_1}([0,T]\times\R^6)$
be a nonnegative classical solution of the Boltzmann equation \eqref{e.boltzmann} on $[0,T]\times\R^6$, with the initial data satisfying
\be
	f_{\rm in} \in L^\infty_{q_2}(\R^6).
\ee
%and $f$ satisfying (CONDITIONAL ASSUMPTIONS) for some $\KO{f}>0$. 
%Additionally, assume 
%\be
%f\in L^\infty_{q_1}([0,T]\times\R^6).% \quad \text{ for some $\theta\in (0, |\gamma|]$.}
%\ee 
Then $f \in L^\infty_{q_2}([0,T]\times \R^6)$, with 
\be
\|f(t)\|_{L^\infty_{q_2}(\R^6)}  \leq \|f_{\rm in}\|_{L^\infty_{q_2}(\R^6)} \exp(C t\|f\|_{L^\infty_{q_1}([0,T]\times\R^6)}), \quad t\in [0,T],
\ee
where the constant $C>0$ depends on $q_1$, $q_2$, $\gamma$, $s$, and $c_b$. 
\end{lemma}

\begin{proof}
The structure of this proof is the same as the proof of Lemma \ref{l:base-estimate}, with a barrier of the form 
\be
g(t,x,v) = N e^{\beta t} \left( \vv^{-q_2} + (\eps\vv)^\alpha + \eps e^{\eps^2\langle x \rangle}\right),
\ee
and $N = (1+\eps)\|f_{\rm in}\|_{L^\infty_{q_2}(\R^6)}$.

The only changes are that we bound $Q(f,\langle \cdot\rangle^{-q_2})(z_0)$ using Lemma \ref{l:Q-polynomial}  instead of Lemma \ref{l:Qs-pointwise-base}, and we bound the term $Q(f,\langle \cdot \rangle^\alpha)$ with the $L^\infty_{q_1}$ norm of $f$, which controls the convolution $f\ast |\cdot|^{\gamma+2s}$ arising in Lemma \ref{l:v-alpha} since $q_1>3+\gamma+2s$,  and also controls the convolution $f\ast |\cdot|^\gamma$ arising in $\Qns$ (see Lemma \ref{l:Qns}). 
%[THIS SHOULD BE FLESHED-OUT; IT IS STILL TRUE THAT $\gamma + 2s - q_1 < -3$ IN ALL PARAMETER REGIMES, BUT NOW YOU MIGHT GET
%$\langle v_0 \rangle^{\alpha+\gamma}$ INSTEAD OF $\langle v_0 \rangle^{\alpha-2s}$]\\
%
This gives, instead of \eqref{e:beta},
\be
\begin{split}
	\beta&\left(\langle v_0\rangle^{-q_2} +(\eps\langle v_0\rangle)^\alpha + \eps e^{\eps^2 \langle x_0\rangle}\right)
		\\&\leq C \langle v_0\rangle^{-q_2} \|f\|_{L^\infty_{q_1}([0,T]\times\R^6)}
			 + C\eps^\alpha \|f\|_{L^\infty_{q_1}([0,T]\times\R^6)} \langle v_0 \rangle^{\alpha+ \gamma}
			 + \eps^2\|f\|_{L^\infty}^{\sfrac1\alpha} e^{\eps^2\langle x_0 \rangle}.
\end{split}
\ee
Therefore, we obtain a contradiction by choosing $\beta=2C\|f\|_{L^\infty_{q_1}([0,T]\times\R^6)} + \eps \|f\|_{L^\infty([0,T]\times\R^6)}^{\sfrac1\alpha}$.  The rest of the argument follows as in \Cref{l:base-estimate} by taking $\eps\to 0$. 
\end{proof}

\subsection{Proving \Cref{t:decay1}}

Next, we prove propagation of pointwise polynomial decay of any order, by starting with the estimate of \Cref{l:base-estimate} and then iteratively applying \Cref{l:gain}. The following theorem is a more detailed statement of our first main result, Theorem \ref{t:decay1}.

\begin{theorem}\label{t:decay-detailed}
Suppose that $q$ and $q_0$ are such that
\be
	q \geq q_0,
	\qquad q_0 \geq 2 + 2s,
	\qquad\text{and}\qquad
	q_0> 3 + \gamma + 2s.
\ee
For  $T>0$, $\eta >0$, assume that  $f\geq 0$ is a bounded classical solution of the Boltzmann equation \eqref{e.boltzmann} on $[0,T]\times\R^6$ with 
\be
\KO{f} < \infty, \quad \|f\|_{L^\infty_{t,x} L^{3/(3+\gamma)+\eta}_v([0,T]\times\R^6)} < \infty, \quad \text{ and } f_{\rm in} \in L^\infty_q(\R^6),
\ee
%\be
%	f \in L^\infty([0,T]\times\R^6),
%		\quad\text{ with } f_{\rm in} \in L^\infty_q(\R^6),
%\ee
%is a  nonnegative classical solution of the Boltzmann equation \eqref{e.boltzmann} on $[0,T]\times\R^6$ with $\KO{f}$ being finite 
where $\KO{f}$ is defined in \eqref{eq:K0_definition}.
%, and for some $q>3 + \gamma + 2s$, assume that 
%\be
%f_{\rm in} \in L^\infty_{q}(\R^6).
%\ee
%Assume that $f\in L^\infty([0,T]\times\R^6)$ and that the quantity $\KO{f}$ defined in \eqref{eq:K0_definition} is finite.
Then $f \in L^\infty_q([0,T]\times\R^6)$, and %\CH{THE $q_0$ BELOW SHOULD BE $q$ NO?}
\be
	\|f(t)\|_{L^\infty_q(\R^6)} \leq \|f_{\rm in}\|_{L^\infty_q(\R^6)} \exp(M_q(t, \|f_{\rm in}\|_{L^\infty_{q}(\R^6)})), \quad t\in [0,T],
\ee
for a function $M_q:\R_+\times\R_+\to \R_+$, increasing in both variables, depending on $q$, $q_0$, $\KO{f}$, $\|f\|_{L^\infty_{t,x}L^{\sfrac{3}{(3+\gamma)} + \eta}_v([0,T]\times\R^6)}$, $\eta$, $T$, $\gamma$, $s$, and $c_b$. 
\end{theorem}

 If we make the choice
\be\label{e.c100301}
	q_0 = \begin{cases}
		3 + \gamma + 2s + \delta
			\qquad&\text{ whenever } \gamma \geq -1,\\
		2+2s
			\qquad&\text{ whenever } \gamma < -1,
		\end{cases}
\ee
for a small $\delta>0$, then \Cref{t:decay1} follows directly from \Cref{t:decay-detailed} since, after choosing $\delta$ sufficiently small, $\KO{f}$ is controlled by the quantities in~\eqref{e.hydro}.
%in the theorem below should ideally be chosen following~\eqref{e.c100301}.  In fact, making the choice~\eqref{e.c100301}, we see that |\Cref{t:decay1} follows directly from \Cref{t:decay-detailed} since, after choosing $\delta$ sufficiently small, $\KO{f}$ is controlled by the quantities in~\eqref{e.hydro}.

\begin{proof}
We proceed by induction by establishing the conclusion for
\be
	q_n = \max\{q, q_0 - n\gamma\}.
\ee
Recall that $-\gamma = |\gamma| > 0$, by assumption.   
The base case $q = q_0$ follows directly from \Cref{l:base-estimate} with the choice
\be
	M_{q_0}(t,z)
		= C t \left(\KO{f} + \|f\|_{L^\infty_{t,x}L^{\frac{3}{3+\gamma} + \eta}_v([0,T]\times\R^6)}\right).
\ee

%If $q > q_0$, we obtain it by iteratively applying \Cref{l:gain} on $[0,t]$ with $q_1 = q_0$ and $q_2 = \max\{q,q_0-\gamma\}$, we obtain
%\be
%\begin{split}
%\|f(t)\|_{L^\infty_{q_0-\gamma}(\R^6)}  &\leq \|f_{\rm in}\|_{L^\infty_{q_0-\gamma}(\R^6)} \exp(C t\|f\|_{L^\infty_{q_0}([0,t]\times\R^6)}),\\
%&\leq \|f_{\rm in}\|_{L^\infty_{q_0-\gamma}(\R^6)} \exp(C t\|f_{\rm in}\|_{L^\infty_{q_0}(\R^6)}\exp(M_{q_0}(t,\|f_{\rm in}\|_{L^\infty_{q_0}(\R^6)}))).
%\end{split}
%\ee
%Let us recall that $- \gamma = |\gamma| > 0$. 
%If $q \leq q_0 - \gamma$, then our conclusions holds with $M_{q_0-\gamma}(t,z) = Ct z \exp(M_{q_0}(t, z))$.

We now show the inductive step, assuming that we have established the inequality for $q_{n-1}$. %it for the case $q \in (q_0 - (n-1)\gamma|, q_0 - n\gamma|]$, where $n \in \N$, assuming that we have already established it for $q_0 - (n-1)\gamma$.  Recall that $-\gamma = |\gamma|>0$, by assumption. 
%If $q > q_0 - \gamma$, we repeat this argument $n_q$ times, where $n_q$ is the smallest natural number such that $q \leq q_0 - n_q\gamma$, with the recursive definition
Let
\be\label{e.M_q}
	M_{q_n}(t,z) = Ct z\exp(M_{q_{n-1}}(t,z))).
\ee 
Applying \Cref{l:gain} on $[0,t]$, % with $q_1 = q_0 - (n-1)\gamma$ and $q_2 = \max\{q,q_0-n\gamma\}$, 
we obtain
\be
\begin{split}
	\|f(t)\|_{L^\infty_{q_n}}
		&\leq \|f_{\rm in}\|_{L^\infty_{q_n}} \exp(C t\|f\|_{L^\infty_{q_{n-1}}([0,t]\times\R^6)}).
%		\\&
%		= \|f_{\rm in}\|_{L^\infty_{q_n}} \exp(C t\|f\|_{L^\infty_{q_0 - (n-1)\gamma}([0,t]\times\R^6)}).
\end{split}
\ee
Then, applying the inductive hypothesis, we find
\be
\begin{split}
	\|f(t)\|_{L^\infty_{q_n}}
		&\lesssim \|f_{\rm in}\|_{L^\infty_{q_n}} \exp(C t \|f_{\rm in}\|_{L^\infty_{q_{n-1}}}\exp(M_{q_{n-1}}(t,\|f_{\rm in}\|_{L^\infty_{q_{n-1}}}))). %\\
%		&= \|f_{\rm in}\|_{L^\infty_{q_n}} \exp(C t \|f_{\rm in}\|_{L^\infty_{q_{n-1}}}\exp(M_{q_{n-1}}(t,\|f_{\rm in}\|_{L^\infty_{q_{n-1}}})))
	\end{split}
\ee
The proof is concluded after using the obvious inequality
\be
	\|f_{\rm in}\|_{L^\infty_{q_{n-1}}}
		\leq \|f_{\rm in}\|_{L^\infty_{q_n}},
\ee
the fact that $M_{q-1}(t,z)$ is increasing in $z$, and the definition~\eqref{e.M_q} of $M_q(t,z)$.
%Note that, for each $n$, we choose $q_1 = q_0 - (n-1)\gamma$ and $q_2 = \max\{q_1 - \gamma, q\}$.  %At the last step, instead of $q_2 = q_1-\gamma$, we may take $q_2 \in (q_1,q_1+|\gamma|]$ as needed in order to obtain the exact $q$ in the statement of the theorem.
%This concludes the proof.
\end{proof}

\section{Decay estimates in terms of weaker quantities: \Cref{t:decay2}}\label{s:better-decay}

\subsection{Previous results and their extensions}

In this section, we combine Theorem \ref{t:decay1} with some ideas from the literature to obtain Theorem \ref{t:decay2}.  First, we recall the self-generating lower bounds of \cite{HST2020lowerbounds}: 

\begin{theorem}{\cite[Theorem 1.2]{HST2020lowerbounds}}\label{t:lower-bounds} 
%\CH{IS THERE SOME ADVANTAGE TO STATING THIS FOR $\gamma$ UP TO 1?}
%Let $\gamma \in (-3,1)$ and $s\in (0,1)$. 
Suppose that $f$ is a nonnegative classical solution of \eqref{e.boltzmann} on $[0,T]\times \R^6$, with initial data $f_{\rm in}$ satisfying the lower bound
\begin{equation*}%\label{e.mass-core}
 f_{\rm in}(x,v) \geq c_m, \quad (x,v) \in B_r(x_m,v_m),
 \end{equation*}
for some $x_m,v_m\in \R^3$ and $c_m, r>0$. %well-distributed with parameters $r$, $\delta$, and $R$. 
Assume that $f$ satisfies %\CH{DO WE NEED POSITIVE PART HERE?}
\begin{equation}\label{e.hydro-general}
\begin{split}
&\sup_{t\in [0,T],x\in \R^3} \int_{\R^3} \vv^{(\gamma+2s)_+} f(t,x,v) \dd v \leq G_0, \qquad \text{ and}\\
&\sup_{t\in [0,T],x\in \R^3} \| f(t,x,\cdot)\|_{L^p(\R^3)} \leq P_0
	\quad \text{ for some } p>\frac{3}{3 + \gamma + 2s} \quad (\mbox{only if } \gamma + 2s < 0).
%\sup_{t,x} \int_{\R^3} |v|^{\gamma+2s} f(t,x,v) \dd v &\leq G_0 \quad (\mbox{only necessary if } \gamma + 2s > 2),
\end{split}
\end{equation}
 Then there are positive, continuous functions $\mu, \eta: (0,T]\times \R^3\to \R$ such that
\be 
f(t,x,v) \geq \mu(t,x-x_m) e^{-\eta(t,x-x_m)|v|^2}, \quad (t,x,v) \in (0,T]\times\R^6.
\ee
Furthermore, $\mu$ and $\eta$ are uniformly positive and bounded on any compact subset of $(0,T]\times \R^3$ and depend only on $T$, $c_m$, $r$, $|v_m|$, $G_0$, and $P_0$. % \CH{THIS USE OF $K_0$ CLASHES WITH THE OTHER ONE}

 If, in addition, $f_{\rm in}$ is well-distributed with parameters $c_m, r, R$, then the functions $\mu$ and $\eta$ can be taken independent of $x$, i.e.
\be
f(t,x,v) \geq \mu(t) e^{-\eta(t) |v|^2}, (t,x,v) \in [0,T]\times\R^6.
\ee
\end{theorem}

 See \eqref{e.well} for the definition of well-distributed initial data.

Next, we have a minor extension of \cite[Theorem 1.2]{silvestre2016boltzmann} %(see also \cite[Theorem 4.1]{imbert2018decay}) 
that does not require the solution $f$ to be periodic in $x$. The result in \cite{silvestre2016boltzmann} derives a pointwise upper bound for $f$ in terms of upper bounds on the mass, energy, and entropy densities, as well as a lower bound on the mass density. It was already pointed out in \cite{HST2020lowerbounds} that Theorem \ref{t:lower-bounds} allows one to remove the lower mass bound and upper entropy bound, at the cost of adding a dependence on $T$ and the initial data.

\begin{proposition}\label{p:Linfty}
Let $f:[0,T]\times\R^6\to [0,\infty)$ be a solution to the Boltzmann equation \eqref{e.boltzmann} that lies in $L^\infty_{q_0}([0,T]\times\R^6)$, and that satisfies the inequalities
\begin{equation}\label{e.hydro2}
\begin{split}
%0 < m_0 \leq 
\int_{\R^3} f(t,x,v)\dd v &\leq M_0,\\
\int_{\R^3} |v|^2 f(t,x,v)\dd v &\leq E_0,\\
%\int_{\R^3} f(t,x,v) \log f(t,x,v) \dd v &\leq H_0,\\
\int_{\R^3} |v|^k f^p(t,x,v) \dd v &\leq P_0,  \quad \text{ (only necessary if $\gamma+2s\leq 0$),}
\end{split}
\end{equation}
uniformly in $(t,x)\in [0,T]\times\R^3$, where $p> \sfrac{3}{(3 + \gamma + 2s)}$ and $k = 1- \sfrac{3 (\gamma+2s)}{(2s)}$. Assume that the initial data $f_{\rm in}$ is well-distributed with parameters $c_m$, $r$, $R$.

Then there exists $N>0$, depending on the constants $M_0$, $E_0$, $P_0$, $p$, $T$, $c_m$, $r$, and $R$, such that
\be
	\|f(t)\|_{L^\infty(\R^6)} \leq N(1+t^{-\frac3{2s}}),  \quad t\in [0,T].
\ee
\end{proposition}

We point out that, importantly, the conclusion is quantitatively independent of the assumption that $f\in L^\infty_{q_0}$, which is the strength of the result. %   \CH{SEE IF WE NEED THIS COMMENT}

\begin{proof}
For $N>0$ to be determined and a small $\eps>0$, define the barrier
\be
	h_\eps(t,x) = N (1 + t^{-\frac{3}{2s}}) e^{\eps\langle x\rangle}.
\ee
This is, essentially, the barrier used in the proof of \cite[Theorem 1.2]{silvestre2016boltzmann}, with an extra factor of $e^{\eps\langle x\rangle}$ included so that $h$ grows as $|x|\to \infty$.  Additionally, the argument here bears strong resemblance to the proof of \Cref{l:base-estimate}, so we merely outline certain steps in the argument.

If $f <  h_\eps$ on $[0,T]\times \R^6$ for every $\eps>0$, then we are finished after taking $\eps$ to zero.  On the contrary, if $f$ is not always smaller than $h_\eps$, we may argue exactly as in the proof of \Cref{l:base-estimate} to obtain
%
%Since $f$ is bounded, we may argue as in the proof of \Cref{l:base-estimate} to obtain 
%
%
%the inequality $f<h$ holds for small time values. If $f<h$ is not true in $[0,T]\times\R^6$, then there is a first time 
a first crossing point $z_0 = (t_0, x_0, v_0)$ with $t_0>0$. % where the inequality is false. Because $f$ lies in $L^\infty_{q_0}$ while $h$ is bounded in $v$ and approaches $+\infty$ as $|x|\to \infty$, there must be a point $z_0 = (t_0,x_0,v_0)$ where $f(z_0) = h(z_0)$ for the first time.
At this crossing point, one has %, as in the proof of Lemma \ref{l:base-estimate},
\begin{equation}\label{e.inequality}
	 \partial_t h_\eps(z_0) + v_0\cdot \nabla_x h_\eps(z_0)
	 	\leq \partial_t f(z_0) + v_0\cdot\nabla_x f(z_0)
		= Q(f,f)(z_0).
\end{equation}
Since $f\in L^\infty_{q_0}$, we have
\be
	h_\eps(z_0)
		= f(z_0) \leq \|f\|_{L^\infty_{q_0}([0,T]\times\R^6)} \langle v_0 \rangle^{-q_0}, 
\ee
and
\be
	|v_0|
		\leq \langle v_0 \rangle
		\leq \left[\frac{\|f\|_{L^\infty_{q_0}([0,T]\times\R^6)}}{h_\eps(z_0)}\right]^{\sfrac{1}{q_0}}.
\ee
Therefore, we can bound the left side of \eqref{e.inequality} from below by
\begin{equation}\label{e.N-lower}
\begin{split}
	N &e^{\eps\langle x_0\rangle}
		\left[ -\frac 3 {2s}  t_0^{-\frac{3}{2s}-1} + \eps v_0\cdot \frac {x_0} {\langle x_0 \rangle}\left(1+t_0^{-\frac{3}{2s}}\right) \right]\\
		& \geq N e^{\eps\langle x_0\rangle}\Bigg[-\frac 3 {2s}  t_0^{-\frac{3}{2s}-1}
			- \eps  \left[\frac{\|f\|_{L^\infty_{q_0}([0,T]\times\R^6)}}{h_\eps(z_0)}\right]^{\sfrac{1}{q_0}}(1+t_0^{-\frac{3}{2s}}) \Bigg]. 
\end{split}
\end{equation}

Next, to bound $Q(f,f)(z_0)$ from above, we can repeat the argument from \cite{silvestre2016boltzmann}, replacing $N$ in their argument with $N e^{\eps\langle x_0\rangle}$. Since the extra factor $e^{\eps\langle x_0\rangle}$ is independent of $v$, it does not interact with the collision operator in any way and can be treated as a constant for the purposes of the upper bound on $Q(f,f)(z_0)$. We do not give the details of the quantitative bounds on $Q(f,f)$ since they are the same as in the proof of \cite[Theorem~7.3]{silvestre2016boltzmann}. In the case $\gamma+2s>0$, we obtain
\begin{equation}\label{e.Qffz}
	Q(f,f)(z_0)
		\leq -c_0 h_\eps(z_0)^{1+\frac{2s}{3}}
				\langle v_0\rangle^{\frac{8s}{3} + \gamma}
			+ C_0 h_\eps(z_0)^{1+\frac{\gamma}{3}},
\end{equation}
for constants $c_0, C_0>0$ depending on the constants in \eqref{e.hydro2}. As written in \cite{silvestre2016boltzmann}, the constant $c_0$ depends also on a lower bound for the mass density and an upper bound on the entropy density $\int f \log f \dd v$.  As pointed out in \cite{HST2020lowerbounds}, the pointwise lower bounds of Theorem \ref{t:lower-bounds} can be used instead of the lower mass bound and upper entropy bound to obtain a ``cone of nondegeneracy'' which is needed to obtain the negative term in \eqref{e.Qffz}. The lower mass bound and upper entropy bound are not used anywhere else in \cite{silvestre2016boltzmann}, so they can be discarded. See \cite[Section 4]{HST2020lowerbounds} for the details.

Combining \eqref{e.Qffz} with \eqref{e.inequality} and \eqref{e.N-lower}, we then have
\be
\begin{split}
	- \frac 3 {2s} N e^{\eps\langle x_0 \rangle}  t_0^{-\frac{3}{2s}-1} &\leq \eps (N e^{\eps\langle x_0\rangle})^{1-\frac{1}{q_0}} (1+t_0^{-\frac{3}{2s}})^{1-\frac{1}{q_0}} \|f\|_{L^\infty_{q_0}}^{\sfrac{1}{q_0}}\\
&\quad - c_0 (N e^{\eps\langle x_0\rangle})^{1+\frac{2s}{3}} (1+t_0^{-\frac{3}{2s}})^{1+\frac{2s}{3}} \langle v_0 \rangle^{\frac{8s}{3}+\gamma}\\
&\quad + C_0 (N e^{\eps\langle x_0\rangle})^{1+\frac\gamma3} (1+t_0^{-\frac{3}{2s}})^{1+\frac\gamma3}.
\end{split}
\ee
Comparing exponents of $N e^{\eps\langle x_0\rangle}$, $t_0$, and $\langle v_0\rangle$, we see that the negative term on the right dominates all other terms if $N$ is large enough, leading to a contradiction. Since $e^{\eps\langle x_0\rangle} \geq 1$, the choice of $N$ can be made independently of $\eps$. 

In the case $\gamma + 2s \leq 0$, the proof of \cite[Theorem~7.3]{silvestre2016boltzmann} gives instead
\be
	Q(f,f)(z_0)
		\leq -c_0 h_\eps(z_0)^{1+\frac{2s p}{3}} + C_0 h_\eps(z_0)^{2-\frac{p(3+\gamma)}{3}},
\ee
and the inequality at $z_0$ reads
\be
\begin{split}
	- \frac 3 {2s} N e^{\eps\langle x_0 \rangle}  t_0^{- \frac{3}{2s}-1}
		&\leq \eps (N e^{\eps\langle x_0\rangle})^{1- \frac{1}{q_0}} (1+t_0^{-\frac{3}{2s}})^{1-\frac{1}{q_0}} \|f\|_{L^\infty_{q_0}}^{\sfrac{1}{q_0}}\\
		&\quad - c_0 (N e^{\eps\langle x_0\rangle})^{1+\frac{2sp}{3}} (1+t_0^{-\frac{3}{2s}})^{1+\frac{2sp}{3}} \\
		&\quad + C_0 (N e^{\eps\langle x_0\rangle})^{2-\frac{p(3+\gamma)}{3}} (1+t_0^{-\frac{3}{2s}})^{2-\frac{p(3+\gamma)}{3}}.
\end{split}
\ee
Because of the choice $p>\sfrac{3}{(3 + \gamma + 2s)}$, the negative term on the right dominates if $N$ is sufficiently large, as in the previous case, leading to a contradiction.

In both cases, we have shown $f< h_\eps$ in $[0,T]\times\R^6$. Sending $\eps\to 0$, we obtain the statement of the proposition.
\end{proof}

Using Proposition \ref{p:Linfty} to bound the $L^\infty_{t,x} L^{3/(3+\gamma)+\eta}_v$ norm of $f$ appearing in our Theorem \ref{t:decay1}, we obtain 
an estimate on $\|f(t)\|_{L^\infty_q(\R^6)}$ that is ``good'' away from $t=0$. Near $t=0$, we use the propagation of $L^\infty_{q_0}$ bounds for a short time depending only on the initial data. The following lemma adapts \cite[Lemma 4.3(a)]{HST2022irregular} to the non-spatially-periodic setting:

\begin{lemma}\label{l:small-times}
For any $q>3 + \gamma + 2s$, and any solution $f\geq 0$ of the Boltzmann equation \eqref{e.boltzmann} on $[0,T]\times\R^6$ such that the initial data $f_{\rm in}$ lies in $L^\infty_{q}(\R^6)$, one has 
\be
\|f(t)\|_{L^\infty_{q}(\R^6)} \lesssim \|f_{\rm in}\|_{L^\infty_{q}(\R^6)} 
\qquad\text{ for all }
t \leq \min\left\{T, \frac{1}{C\|f_{\rm in}\|_{L^\infty_{q}(\R^6)}}\right\},
\ee
where $C>0$ is a constant depending on $q$, $\gamma$, $s$, and $c_b$.
\end{lemma}

\begin{proof}
Define $H(t) = \|f\|_{L^\infty_{q}([0,t],\R^6)}$. From Lemma \ref{l:gain} with $q_1 = q_2 = q$, we have $H(t) \leq H(0) e^{Ct H(t)}$ for all $t\in [0,T]$. We complete the proof by applying \cite[Lemma 4.2]{HST2022irregular}, which states that for any continuous, increasing function $H:[0,T]\to \R_+$ satisfying $H(t) \leq A e^{BtH(t)}$ for some $A,B>0$, there holds $H(t) \leq eA$ up to time $\min\{T, (eAB)^{-1}\}$.
\end{proof}

\subsection{Proof of \Cref{t:decay2}}

%We are now ready to prove Theorem \ref{t:decay2}. 
\begin{proof}
Lemma \ref{l:small-times} gives a bound on $\|f(t)\|_{L^\infty_q(\R^6)}$ that depends only on the $L^\infty_q$-norm of the initial data, as long as $t\leq C/\|f_{\rm in}\|_{L^\infty_q(\R^6)}$. For $t > C/\|f_{\rm in}\|_{L^\infty_q(\R^6)}$, we apply Proposition \ref{p:Linfty} to obtain a bound on $\|f(t)\|_{L^\infty(\R^6)}$ depending on the constants in the statement of Theorem \ref{t:decay2}. By interpolating this $L^\infty$ bound with the uniform bound $\int_{\R^3} f(t,x,v) \dd v \leq M_0$, we remove the dependence on the $L^\infty_{t,x}L^{3/(3+\gamma)+\eta}_v$ norm appearing in Theorem \ref{t:decay1}. This establishes Theorem \ref{t:decay2}.
\end{proof}

%\section{Continuation criterion}

%\section{Proof of the continuation criterion \Cref{c:cont}}
%\stan{Maybe move the following proof to the intro, right after the statement of \Cref{c:cont}?} \CH{I'M FINE EITHER WAY}

\subsection{Proof of \Cref{c:cont}}

In this subsection, we use Theorem \ref{t:decay2} to derive the continuation criterion of \Cref{c:cont}. 

\begin{proof}
%Finally, we prove the continuation criterion, \Cref{c:cont}.
We first recall the local existence theorem of \cite[Theorem 1.1]{HST2022irregular}, which states that for any $f_{\rm in} \in L^\infty_q(\R^6)$ with $q>3+2s$ satisfying a uniformly positive lower bound $f_{\rm in}(x,v)\geq c_m$ in some ball in $(x,v)$ space, there exists a classical solution $f\geq 0$ of the Boltzmann equation on $[0,T_f]\times\R^6$ with $f(0,x,v) = f_{\rm in}$. Importantly, this time of existence $T_f$ depends only on $q$ and $\|f_{\rm in}\|_{L^\infty_q(\R^6)}$, and does not depend quantitatively on the lower bounds satisfied by $f_{\rm in}$. 

Given any solution on $[0,T]\times\R^6$ satisfying the conditional bounds \eqref{e.conditional2}, with initial data well-distributed and lying in $L^\infty_q(\R^6)$, Theorem \ref{t:decay2} states that the $L^\infty_q(\R^6)$ norm of $f$ remains bounded by some finite constant $K$ on $[0,T]$. The necessary lower bounds are propagated to time $T$ via %\cite[Theorem 1.2]{HST2020lowerbounds} (which is quoted below as Theorem \ref{t:lower-bounds} in the current article) 
Theorem \ref{t:lower-bounds} 
and the upper bound from \eqref{e.conditional2}. Therefore, we can apply \cite[Theorem 1.1]{HST2022irregular} with initial data $f(T,x,v)$ to extend the solution past $T$. We conclude $T$ cannot be the maximal time of existence.
\end{proof}

\section{The main estimate on the collision kernel: \Cref{l:Qs-pointwise-base}}%Iterative pointwise bounds on $Q_{\rm s}$}
\label{s.main_lemma}

In this section, we establish our main estimate \Cref{l:Qs-pointwise-base}, as well as \Cref{l:v-alpha}. 
%prove two pointwise bounds on $Q_{\rm s}(f, \langle \cdot \rangle^{-q})$ for various values of $q$. The first will establish a
%base case with $q=4-\gamma$ that only depends on an upper bound for the mass and energy density of $f$. The second will iterate by assuming
%that $f$ also satisfies an $L^\infty_k$ weighted pointwise bound, which will allow $q$ to become larger.
For the first result, we make use of the density of ``simple'' functions in $L^1_2$, defined as follows.
\begin{definition}\label{d:simple}
Given a fixed $\underline v \in \R^3$, we say that a function $f: \R^3 \rightarrow \R$ is {\bf simple with base point $\underline v$} if there is a positive integer $m$, a list of disjoint open balls $\lbrace B_{r_i}(v_i) \rbrace_{i=1}^m$ such that
\be
	0 < r_i < \min\Big(1, \frac{|\underline v-v_i|}{|\underline v|}\Big),
\ee
and a list of values $\lambda_1, ... , \lambda_m$ such that
\be
	f(v) = \sum_{i=1}^m \lambda_i \1_{B_{r_i}(v_i)}(v).%\1_{B_{r_i}(v_i)}(v).
\ee
\end{definition}
 Our definition of ``simple'' is slightly different than is standard in analysis: we use only indicator functions of balls and each ball has a restriction on its radius.  As such, it is not a linear space.  However, our class of simple functions is dense in $L^p_k(\R^3)$:
\begin{proposition}\label{p:simple-L12-density}
Fix any $\underline v\in \R^3$.  The set of all simple functions with base point $\underline v$ forms a dense subset of $L^p_k(\R^3)$ for any $1 \leq p < \infty$ and any $k$. %[FUN FACT: IT'S NOT A SUBSPACE SINCE YOU CAN'T ADD 'EM]
\end{proposition}
 The proof of Proposition \ref{p:simple-L12-density} follows standard methods and is omitted.

In \Cref{d:simple} above, it is important that the individual terms in the sum have disjoint supports as this allows us to write
\be
	\| f \|_{L^p_k(\R^3)}^p
		= \sum_{i=1}^m \| \lambda_i \1_{B_{r_i}(v_i)} \|_{L^1_k(\R^3)}^p
		\approx \sum_{i=1}^m |\lambda_i|^p \langle v_i \rangle^k r_i^3,
\ee
where the constants implied by the ``$\approx$'' symbol do not depend on $p$, $k$, the $v_i$, the $r_i$, or (importantly) on $m$. 
This precision will be needed for the proof of Lemma \ref{l:Qs-pointwise-base}. % base case result below. %[NOT-SO-FUN FACT: THIS IS THE MAIN REASON WHY A SIMILAR APPROXIMATION WON'T WORK FOR $L^\infty_k$-SPACES]

\subsection{Proof of \Cref{l:Qs-pointwise-base}}\label{s:Qs-pointwise-base-proof}
We begin by first stating a weaker lemma that will be bootstrapped up to \Cref{l:Qs-pointwise-base}.  

\begin{lemma}\label{l:preliminary_Qs}
	Under the assumptions of \Cref{l:Qs-pointwise-base}, we have
	\be\label{e:dfjk693}
		Q_s(f,\langle \cdot\rangle^{-q_0})(v)
			\lesssim \vv^{2 + \gamma} \KO{f}
			\qquad\text{ for all } v.
	\ee
	Recall the definition \eqref{eq:K0_definition} of $\KO{\cdot}$.
\end{lemma}

The importance of \Cref{l:preliminary_Qs} is that it allows us to, by density, first prove \Cref{l:Qs-pointwise-base} on simple functions (defined in \Cref{d:simple}).  Its proof is in \Cref{s.preliminary_Qs}.  With this reduction in hand, we now proceed with the key estimate of \Cref{l:Qs-pointwise-base}.

\begin{proof}[Proof of~\Cref{l:Qs-pointwise-base}]
Fix any $v \in \R^3$.  Without loss of generality, we may assume that
\be\label{e.v_0}
	v = (0,0,|v|).
\ee
This will be convenient for some of the geometry, below.  Observe that, if $|v| \leq 10$, the very rough bound in \Cref{l:preliminary_Qs} is sufficient to prove \eqref{eq:Qs-pointwise-base}, since $\vv^{-q_0} \approx \vv^{2+\gamma} \approx 1$ on a fixed compact set. For the rest of the proof, we shall assume that $|v| > 10$, which also means that $|v| \approx \vv$. %All implied constants are universal.  

%Now, let $f$ be as in the hypothesis of the lemma. Without loss of generality, we may assume that $v = (0,0,|v|)$; $v$ is now fixed. 
Choose a simple function with base point $v$,
\be\label{e.tilde_f}
	\tilde f (w) = \sum_{i=1}^m \lambda_i \1_{B_{r_i}(v_i)}(w),
\ee
such that, by \Cref{p:simple-L12-density},
\begin{equation}\label{e.K_0_approximation}
	\KO{f-\tilde f} \leq \vv^{-q_0-2 - \gamma} \KO{f}.
\end{equation}
%This is possible due to \Cref{l:preliminary_Qs}. 
We may also assume that each $\lambda_i$ is nonnegative.  Indeed, due to the nonnegativity of $f$, the approximation~\eqref{e.K_0_approximation} only becomes better when we remove all negative terms in the sum~\eqref{e.tilde_f}. Then, by Lemma \ref{l:preliminary_Qs} and \eqref{e.K_0_approximation},
\begin{equation}\label{eq:Qs-simple-approx}
\begin{split}
Q_{\rm s}(f, \langle \cdot \rangle^{-q_0})(v)
	&= Q_{\rm s}(f-\tilde f, \langle \cdot \rangle^{-q_0})(v) + Q_{\rm s}(\tilde f, \langle \cdot \rangle^{-q_0})(v) \\
	&\lesssim \langle v \rangle^{2+ \gamma} \KO{f-\tilde f} +
Q_{\rm s} \Big( \sum_{i=1}^m \lambda_i \1_{B_{r_i}(v_i)}, \langle \cdot \rangle^{-q_0} \Big)(v) \\
	&\leq \vv^{-q_0}\KO{f} + \sum_{i=1}^m \lambda_i Q_{\rm s}(  \1_{B_{r_i}(v_i)}, \langle \cdot \rangle^{-q_0})(v).
\end{split}
\end{equation}
Let us point out that, if we can establish that, for all $i$,
\be\label{e.c100201}
	Q_s(\1_{B_{r_i}(v_i)}, \vvd^{-q_0})
		\lesssim \vv^{-q_0} \KO{\1_{B_{r_i}(v_i)}},
\ee
we are finished.  Indeed, putting together~\eqref{eq:Qs-simple-approx} and~\eqref{e.c100201} and recalling~\eqref{e.K_0_approximation}, we find
\be
	\begin{split}
		\Qs(f, \langle \cdot \rangle^{-q_0})(v)
			&\lesssim \vv^{-q_0}\KO{f} + \vv^{-q_0} \sum_{i=1}^m \KO{\lambda_i \1_{B_{r_i}(v_i)}}
			\\&
			\approx \vv^{-q_0}\KO{f} + \vv^{-q_0} \KO{\tilde f}
			\\&
			\lesssim \vv^{-q_0} \left( \KO{f} + \KO{\tilde f - f} + \KO{f} \right)
			\leq 3 \vv^{-q_0} \KO{f}.
	\end{split}
\ee

The remainder of the proof is devoted to establishing~\eqref{e.c100201}. 
%From now on, we write $K_0$ in place of $\KO{f}$.  
%Each term in the sum above %part of the collision operator is essentially 
%is given by 
First,  using formula \eqref{e.Qs} for $Q_{\rm s}$, we decompose $\Qs$ into three integrals: 
\begin{equation}\label{eq:Qs-for-simple}
\begin{split}
	Q_{\rm s}&(\1_{B_{r_i}(v_i)}, \langle \cdot \rangle^{-q_0})(v)
		= \int_{\R^3} (\vvp^{-q_0} - \vv^{-q_0}) K_{\lambda_i \1_{B_{r_i}(v_i)}}(v,v') \dd v'
%		\\&
%		\approx \int_{\R^3} \frac{\langle v' \rangle^{-q_0}-\vv^{-q_0}}{|v-v'|^{3+2s}} \int_{w \perp (v-v')} \lambda_i \1_{B_{r_i}(v_i)}(v+w) |w|^{\gamma+2s+1} \dd w \dd v'
		=: I_1 + I_2 + I_3.
\end{split}
\end{equation}
Here we have defined $I_1$, $I_2$, and $I_3$ as the integrals on the regions
\be
	R_{\rm sing} = B_{\sfrac{|v|}4}(v),
	\quad
	R_0 = B_{\sfrac{|v|}{4}}(0),
	\quad\text{ and }\quad
	R_\infty = \R^3 \setminus (R_{\rm sing} \cup R_0),
\ee
respectively. 
Roughly, $R_{\rm sing}$ captures the singularity at $v'=v$, $R_0$ is the region around $v'=0$ where we cannot lean on the smallness of $\vvd^{-q_0}$, and $R_\infty$ is the unbounded region.

Let us prepare for this argument by proving a useful bound regarding $w$ that holds when $v+w \in B_{r_i}(v_i)$.  We claim that
\be\label{e.c092802}
	|w| \approx |v-v_i|.
\ee
Indeed, recalling the restriction on $r_i$ in \Cref{d:simple} and that $|v| \geq 10$, we find that
\be
	\begin{split}
		|w| &= |(v + w - v_i) - (v-v_i)|
			\leq |v+w-v_i| + |v-v_i|
			< r_i + |v-v_i|
			\\&
			\leq \frac{|v-v_i|}{|v|} + |v-v_i|
			\lesssim |v-v_i|.
	\end{split}
\ee
 For the matching lower bound, we similarly have
\be
	\begin{split}
		|w| &	\geq |v-v_i| - |v+w-v_i| \\
			&\geq |v-v_i| -  r_i 
			\\&
			\geq |v-v_i| - \frac{|v-v_i|}{|v|} 
			\gtrsim |v-v_i|,
	\end{split}
\ee
since $1-1/|v| \geq \frac 9 {10}$.  Thus,~\eqref{e.c092802} is established.

The integrals $I_1$ and $I_3$ are the easiest to bound, so we start our argument by estimating these.  Consider $I_1$, the integral over the singular region $R_{\rm sing}$.  Taylor's theorem implies that
\be \label{e.taylor}
	\vvp^{-q_0} - \vv^{-q_0} = -q_0 v \vv^{-q_0-2}(v'-v) + E(v,v')
\ee
where, in $R_{\rm sing}$,
\be
	|E(v,v')| \lesssim \vv^{-q_0-2} |v-v'|^2.
\ee
%to write $\langle v' \rangle^{-q_0} - \vv^{-q_0} = (\gamma-4) v \vv^{-q_0-2} + E(v,v')$ where $|E(v,v')| \lesssim \vv^{-q_0-2} |v-v'|^2$ for all $v' \in B_{|v|/4}(v)$. 
By the symmetry of the collision kernel  ($K_f(v,v+w) = K_f(v,v-w)$), the first order term in the expansion \eqref{e.taylor} vanishes from the integral $I_1$. Lemma \ref{l:K-upper-bound-2} then implies
\begin{equation}\label{eq:Qs-simple-I1}
\begin{split}
	|I_1|
		&\lesssim \vv^{-q_0-2}
			\int_{B_{|v|/4}(v)}
				|v-v'|^2 K_{ \1_{B_{r_i}(v_i)}}(v,v') \dd v'
		\\&
		\lesssim \vv^{-q_0}|v|^{-2s}
			\int _{\R^3}  \1_{B_{r_i}(v_i)}(v+w) |w|^{\gamma+2s} \dd w
		\\&
 \lesssim \vv^{-q_0-2s} \vv^{(\gamma+2s)_+} \KO{\1_{B_{r_i}(v_i)}} \lesssim \vv^{-q_0+\max\{\gamma,-2s\}} \KO{\1_{B_{r_i}(v_i)}}.
%		\approx \vv^{-q_0-2s}
%			\int_{B_{r_i}(v_i)}  |v-v_i|^{\gamma+2s} \dd w
%		\approx \vv^{-q_0-2s}  r_i^3 |v_i - v |^{\gamma+2s}.
\end{split}
\end{equation}
 In the second-to-last step, we used the convolution estimate of Lemma \ref{l:convolution} with $f = \1_{B_{r_i}(v_i)}$. (To apply Lemma \ref{l:convolution}, we used $1 \lesssim \ww^{q_0-2-2s}$.) %~\eqref{e.c092802}.

Next, we consider $I_3$, the integral over the unbounded region $R_\infty$.  On this domain, $\langle v' \rangle^{-q_0} \lesssim \vv^{-q_0}$. By Lemma \ref{l:K-upper-bound}, we get %(following a similar analysis for $I_1$)
\begin{equation}\label{eq:Qs-simple-I3}
\begin{split}
	|I_3|
		&\lesssim \vv^{-q_0} \int_{R_\infty} K_{ \1_{B_{r_i}(v_i)}} (v,v') \dd v'
		\leq \vv^{-q_0} \int_{B_{\sfrac{|v|}{4}}(v)^c} K_{ \1_{B_{r_i}(v_i)}} (v,v') \dd v'
		\\&
%		= \vv^{-q_0} \sum_{k=-2}^\infty \int_{B_{2^{k+1}|v|}(v) \setminus B_{2^{k}|v|}(v)} K_{ \1_{B_{r_i}(v_i)}} (v,v') \dd v'
	%	\\&
		\lesssim \vv^{-q_0}   |v|^{-2s} \int_{\R^3}  \1_{B_{r_i}(v_i)}(v+w) |w|^{\gamma+2s} \dd w
		\\&
		 \lesssim \vv^{-q_0+\max\{\gamma,-2s\}} \KO{\1_{B_{r_i}(v_i)}},
	%	\approx \vv^{-q_0} |v|^{-2s} \int_{B_{r_i}(v_i)}  |v-v_i|^{\gamma+2s} \dd w 
	%	 \approx \vv^{-q_0-2s}  r_i^3 |v - v_i|^{\gamma+2s}.
\end{split}
\end{equation}
using Lemma \ref{l:K-upper-bound} exactly as in \eqref{eq:Qs-simple-I1}. %In the second-to-last step, we used~\eqref{e.c092802} again.

%Although it is not yet clear, we should consider $|v-v_i|^{\gamma+2s} \sim \vv^{\gamma+2s}$, and we see that
%the pointwise decay seen in $I_1$ and $I_3$ has actually \emph{improved} from $-q_0$ to $-q_0 +\gamma$.  The estimate of $I_2$ will not be as profitable.

The estimate of $I_2$ involves a somewhat complicated analysis of the vectors $v_i$, $v'$, and $v$.  Let us begin by defining the set of $v'$ for which the $w$-integral in the definition \eqref{e.Qs} of $K_f$ is nontrivial; that is all $v'$ such that $B_{r_i}(v_i)$ intersects $v + (v-v')^\perp$.  Let
\be \label{e:Didef}
	\mathcal{D}_i := \lbrace v' \in R_0 : \text{ there is some $w \perp (v-v')$ such that $v+w \in B_{r_i}(v_i)$ } \rbrace.
\ee
The key estimates are: 
\be\label{e.c100501}
	\int_{\cD_i} \vvp^{-q_0} \dd v'
		\lesssim \frac{|v| r_i}{|v_i - v|}
\ee
and, whenever $\cD_i \neq \emptyset$,
\be\label{e.c092903}
	|v_i - v|
		\approx \langle v_i\rangle
		\approx | v_i |
		\geq \frac{|v|}{4}.
\ee
%{\color{blue} Heuristically, \eqref{e.c092903} says that $|v|$ cannot be too large compared to $|v_i|$, if }
To aide the reader, we postpone establishing~\eqref{e.c100501}-\eqref{e.c092903} until the end of the proof.  We now proceed with the estimate of $I_2$.

We begin the bound of $I_2$ by noting that, since $\cD_i \subset R_0$, 
\be
	\vvp^{-q_0} > \vv^{-q_0}
		\quad\text{ and }\quad
	|v-v'| \approx |v| \approx \vv
	\qquad\text{ for all } v' \in \cD_i.
\ee
Hence,
\be
\begin{split}
	|I_2|
		&\approx \int_{\cD_i} \frac{\langle v' \rangle^{-q_0}-\vv^{-q_0}}{|v-v'|^{3+2s}}
\int_{w \perp (v-v')}  \1_{B_{r_i}(v_i)} (v+w) |w|^{\gamma+2s+1} \dd w \dd v' 
		\\&
		\lesssim \int_{\cD_i} \frac{\langle v' \rangle^{-q_0}}{\vv^{3+2s}}
\int_{w \perp (v-v')}  \1_{B_{r_i}(v_i)} (v+w) |w|^{\gamma+2s+1} \dd w \dd v'.
\end{split}
\ee
Then, recalling~\eqref{e.c092802}, we find
\be\label{e.I2elements}
\begin{split}
	|I_2|
		&\lesssim \int_{\cD_i} \frac{\langle v' \rangle^{-q_0}}{\vv^{3+2s}}
\int_{w \perp (v-v')}  \1_{B_{r_i}(v_i)} (v+w) |v-v_i|^{\gamma+2s+1} \dd w \dd v' 
		\\&
		\lesssim \int_{\cD_i} \frac{\langle v' \rangle^{-q_0}}{\vv^{3+2s}}  r_i^2 |v_i - v|^{\gamma+2s+1} \dd v'
		\approx \frac{ r_i^2}{\vv^{3+2s}} |v_i-v|^{\gamma+2s+1} \int_{\cD_i} \langle v' \rangle^{-q_0} \dd v'.
\end{split}
\ee
We now apply~\eqref{e.c100501}, to deduce that
\be
	|I_2|
		\lesssim \frac{ r_i^3}{\vv^{2+2s}} |v_i-v|^{\gamma+2s}
		= r_i^3 \vv^{-q_0} \vv^{q_0 - 2 -2s} |v_i-v|^{\gamma+2s}.
\ee
Using that $q_0\geq 2+2s$, by assumption, and applying~\eqref{e.c092903} yields
		\be\label{e.cI2}
				|I_2| 
					 \lesssim  r_i^3 \vv^{-q_0} \langle v_i \rangle^{q_0 - 2 -2s} |v_i - v|^{\gamma+2s}.
		\ee
	%	\be
		%	\begin{split}
		%	\Qs( \1_{r_i}(v_i), \vvd^{-q_0})
		%		&\leq |I_1| + |I_2| + |I_3|
		%		\lesssim \vv^{-q_0}  r_i^3 \langle v_i \rangle^{q_0 - 2-2s} |v-v_i|^{\gamma + 2s}.
		%	\end{split}
		%\ee
	Applying~\eqref{e.c092903} once again, we deduce that
		\be
			\begin{split}
			%\Qs( \1_{r_i}(v_i), \vvd^{-q_0})
			 |I_2|	&\lesssim \vv^{-q_0}  r_i^3 \langle v_i \rangle^{q_0 - 2-2s} \vvi^{\gamma + 2s} \\
				&\approx \vv^{-q_0} \int_{\R^3} \ww^{q_0 - 2+\gamma} \1_{B_{r_i}(v_i)}(w) \dd w
				\\&
				\lesssim \vv^{-q_0} \|\1_{B_{r_i}(v_i)}\|_{L^1_{q_0-2+\gamma}}
				\lesssim \vv^{-q_0} \KO{\1_{B_{r_i}(v_i)}}.
%				\approx \vv^{-q_0} \int \ww^{q_0 - 2 - 2s} \1_{B_{r_i}(v_i)}(w) |v-w|^{\gamma + 2s} \dd w.
			\end{split}
		\ee	
Combining this with the estimates of $I_1$ and $I_3$, given by~\eqref{eq:Qs-simple-I1} and~\eqref{eq:Qs-simple-I3}, respectively, we have shown~\eqref{e.c100201}, which concludes the proof, up to establishing~\eqref{e.c100501}-\eqref{e.c092903}.

	We first prove~\eqref{e.c092903}.  To begin, we characterize the set of points that lie on a plane of the form $v + (v-v')^\perp$ for some $v' \in R_0$:  %determine the possible values $v_i$ and $v'$ for which the $w$-integral of \eqref{eq:Qs-for-simple} be nonzero. 
	define
	\be
		\mathcal{AC}
			:= \lbrace (u_1, u_2, u_3) \in \R^3 : 15 (u_3-|v|)^2  < u_1^2 + u_2^2 \rbrace .
	\ee
	Here we use the notation $\mathcal{AC}$ for ``anti-cone,'' as this region is the complement of a (double) cone in $\R^3$ with apex at $v$. %two cones, each the reflection of the other, whose vertex meets at $v$.  
	
	Let us justify our heuristic description of $\mathcal{AC}$.  We claim that $v+w \in \mathcal{AC}$ whenever $v' \in R_0 = B_{\sfrac{|v|}{4}}(0)$ and $w \perp (v-v')$.  Indeed, $w\perp (v-v')$ implies that
\be
	|v| w_3 = v\cdot w = v' \cdot w,
\ee
 since we have chosen coordinates such that $v = (0,0,|v|)$. 
Hence, using Cauchy-Schwarz and the condition $|v'| \leq \sfrac{|v|}{4}$, we deduce that
\be
	|w_3| \leq \frac{|w|}{4}.
\ee
Multiplying the $4$ to the left hand side and squaring this, we find
\be
	16w_3^2 \leq w_1^2 + w_2^2 + w_3^2.
\ee
After subtracting $w_3^2$ from both sides and writing $(w_1, w_2, |v| + w_3) = v+w = (u_1,u_2,u_3)$, we deduce
\be
	15 (u_3 - |v|)^2
		= 15 w_3^2
		\leq w_1^2 + w_2^2
		= u_1^2 + u_2^2;
\ee
that is, $v+w \in \mathcal{AC}$.
\AC %% COMMAND TO MAKE FIGURE OF THE ANTI CONE

One can also see this geometrically using the fact that %, we must necessarily have $v+w \in \mathcal{AC}$. The number $15$ in the definition is an optimal ratio that comes from the choice of radius for the ball centered at the origin; specifically, 
$\sfrac{1}{15} = \tan^2 (\arcsin(\sfrac14))$.  Indeed, $\mathcal{AC}$ is the rotation around the $x_3$ axis of a two dimensional cone in the $x_2$-$x_3$ plane with angle $\arcsin(\sfrac14)$, apex at $v$, and oriented along the $x_2$-axis.  Alternatively, it is the complement of a cone (in all variables) of angle opening $\sfrac\pi2 - \arcsin(\sfrac14)$, apex at $v$, and oriented along the $x_3$-axis. 
We note here that working on the smaller ball of radius $\sfrac{|v|}{4}$ is crucial to this argument.  The $4$ in $\arcsin(\sfrac14)$ is the same one as in $\sfrac{|v|}{4}$.  For example, if we had worked with a ball of radius $\sfrac{|v|}{1+\eps}$, we would have $\tan^2(\arcsin(\sfrac{1}{1+\eps}))\approx \sfrac{1}{2\eps}$ and all constants would blow up.

If $\cD_i$ defined as in \eqref{e:Didef} is the empty set, then $B_{r_i}(v_i) \cap \mathcal{AC} = \emptyset$ and $I_2 = 0$. We may ignore this case.  If $\cD_i \neq \emptyset$, then $B_{r_i}(v_i) \cap \mathcal{AC} \neq \emptyset$, which implies that $v_i$ must be within a $r_i$-neighborhood of $\mathcal{AC}$.  Recall, from \Cref{d:simple}, that $r_i < 1$.  Hence,
\be\label{e.c100502}
	\inf_{u \in \mathcal{AC}} |v_i - u|
		< r_i
		< 1.
\ee

We see that~\eqref{e.c092903} is more or less obvious from \Cref{f.AC}, but let us prove it algebraically.  Notice that both of the $\approx$ in \eqref{e.c092903} follow from the inequality $|v_i|\geq |v|/4$, since $|v| \geq 10$.  Therefore, it suffices to show this inequality.  %Moreover, by the triangle inequality, it is enough to simply show that $|v_i| \gtrsim |v|$.  
To establish this, fix $u \in \mathcal{AC}$.  % it is enough to show it for every element in $u \in \mathcal{AC}$ with an implied constant of $\sfrac12$, since $v_i$ is in a 1-neighborhood of $\mathcal{AC}$ and $|v| \geq 10$. 
Either $u_3 \geq \sfrac{|v|}{2}$ or we have
\be
	\frac{15}{4} |v|^2
		= 15 \left(\frac{|v|}{2}\right)^2
		\leq 15 (|v| - u_3)^2
		\leq u_1^2 + u_2^2
		\leq |u|^2.
\ee
In either case, using~\eqref{e.c100502} and the fact that $|v| \geq 10$, we deduce
\be
	|v_i|
		\geq \sup_{u \in \mathcal{AC}} (|u| - |v_i - u|)
		\geq \frac{|v|}{2} - 1
		\geq \frac{|v|}{4}.
\ee
This concludes the proof of~\eqref{e.c092903}.

We now consider~\eqref{e.c100501}.  Roughly, since $r_i$ is ``small'' compared to $|v_i-v|$,  $\cD_i$ is a narrowing ``slice'' %(disc-shaped) 
of the ball $R_0 = B_{\sfrac{|v|}{4}}$ with thickness proportional to $\sfrac{|v|r_i }{|v_i-v|} < 1$.  See \Cref{f.cD} for a depiction of this.
	\DI
Let us make this precise.  %Up to rotation, we may assume that $v_i = (v_{i,1}, 0, v_{i,3})$.  (We may apply any rotation that leaves $v$ fixed, hence, the third coordinate in $v_i$ must remain fixed; recall~\eqref{e.v_0}). 
We claim that there is a plane $P_i$ such that
\be\label{e.c100404}
	\cD_i
		\subset v + P_i + \left(-\frac{C|v| r_i}{|v_i-v|},\frac{C|v| r_i}{|v_i-v|}\right) n
\ee
where $n$ is the unit normal vector to $P_i$.  Taking~\eqref{e.c100404} for granted momentarily, we immediately see that there is a rotation $E$ and a shift $\sigma_i$ such that
\be
	E\cD_i
		\subset \left(-\frac{C|v| r_i}{|v_i-v|} +\sigma_i,\frac{C|v| r_i}{|v_i-v|}+ \sigma_i\right) \times B^{(2)}_{\sfrac{|v|}{4}},
\ee
%
%see that there is change of basis matrix $E$ sending $n$ to $(1,0,0)$ and $P$ to the plane spanned by $(0,1,0)$ and $(0,0,1)$.  Hence,
%\be
%	E(\cD_i - v)
%		\subset \left(-\frac{C|v| r_i}{|v_i-v|} -\sigma_i,\frac{C|v| r_i}{|v_i-v|}\right) \times B^{(2)}_{\sfrac{|v|}{4}},
%\ee
where we write $B^{(2)}_{\sfrac{|v|}{4}}$ for the two-dimensional ball of radius $\sfrac{|v|}{4}$ around the origin.  Indeed, take $E$ to be the rotation that takes $P_i$ to the plane spanned by $(0,1,0)$ and $(0,0,1)$.  See \Cref{f.cD} for a depiction of this. 
Hence, we deduce~\eqref{e.c100501} immediately using the fact that $q_0 \geq 2 + 2s> 2$, by assumption:
\be
\begin{split}
	\int_{\cD_i} \vvp^{-q_0} \dd v'
		&= \int_{E\cD_i} \vvp^{-q_0} \dd v'
		\leq 
			\int_{-\frac{C|v| r_i}{|v_i-v|} +\sigma_i }^{\frac{C|v| r_i}{|v_i-v|}+\sigma_i}\Big(\int_{B^{(2)}_{|v|/4}} \langle v' \rangle^{-q_0} \dd v'_2\dd v_3'\Big) \dd v_1'
		\\&
		\leq \int_{-\frac{C|v| r_i}{|v_i-v|}}^{\frac{C|v| r_i}{|v_i-v|}}\Big(\int_{B^{(2)}_{|v|/4}} \langle (0,v_2',v_3') \rangle^{-q_0} \dd v'_2\dd v_3'\Big) \dd v_1'
		\\&
		\lesssim \int_{-\frac{C|v| r_i}{|v_i-v|}}^{\frac{C|v| r_i}{|v_i-v|}} \dd v_1'
		\lesssim \frac{|v| r_i}{|v_i-v|}.
\end{split}
\ee
In the first step above, we used that $\langle Ev'\rangle = \langle v'\rangle$, since $E$ is a rotation.

To see that~\eqref{e.c100404} holds, take $P_i = (v-v_i)^\perp$ and write $v' - v = p + \alpha n$, where $p\in P_i$ and $n$ is the unit normal vector to $P_i$ given by
\be\label{e.c100504}
	n = \frac{v-v_i}{|v-v_i|}.
\ee
Then~\eqref{e.c100404} is equivalent to
\be\label{e.c100503}
	|\alpha| \lesssim \frac{|v| r_i}{|v_i-v|}.
\ee
To verify \eqref{e.c100503}, notice that
\be\label{e.c100505}
	|\alpha| |v_i-v|
		= |(p + \alpha n)\cdot (v_i-v)|
		= |(v'-v)\cdot (v_i-v)|.
\ee
The first equality uses~\eqref{e.c100504}.  By definition of $\cD_i$, there is 
\be
	u \in \left(v + (v-v')^\perp\right) \cap B_{r_i}(v_i).
\ee
Using this in~\eqref{e.c100505}, we find
\be
	\begin{split}
		|\alpha| |v_i-v|
			&= |(v'-v) \cdot (v_i - u + u - v)|
			\leq |(v'-v)||(v_i-u)| + |(v'-v)\cdot(u-v)|
			\\&
			=|(v'-v)||(v_i-u)|
			\leq (|v'|+|v|) r_i
			\leq \frac{5|v|}{4} r_i.
	\end{split}
\ee
In the last inequality, we used that $|v'|\leq \sfrac{|v|}{4}$, by definition of $R_0$.  This is precisely~\eqref{e.c100503}.  This concludes the proof of~\eqref{e.c100501} and, thus, the lemma.
\end{proof}

\subsection{Proof of the weaker estimate \Cref{l:preliminary_Qs}}\label{s.preliminary_Qs}

We now demonstrate that the coarse initial estimate \eqref{e:dfjk693} holds.

%We are now in position to establish \Cref{l:preliminary_Qs}, the preliminary estimate on $\Qs(f,\vvd^{-q_0})$.

\begin{proof}[Proof of \Cref{l:preliminary_Qs}]
 First, if $|v|< 10$, we can apply Lemma \ref{l:Qs-interp} with $g(v) = \vv^{-q_0}$ to obtain
\be
|Q_{\rm s}(f,\langle \cdot\rangle^{-q_0})| \lesssim \int_{\R^3} |w|^{\gamma+2s} f(v+w) \dd w,
\ee
since $g$ and $D_v^2g$ are uniformly bounded on $\R^3$. The convolution estimate of Lemma \ref{l:convolution} then implies
\be
|Q_{\rm s}(f,\langle \cdot\rangle^{-q_0})| \lesssim \vv^{(\gamma+2s)_+}\KO{f} \lesssim \vv^{\gamma+2} \KO{f},
\ee
as desired. 
%We consider only the case when $|v| \geq 10$.  When $|v| \leq 10$, the adjustments are straightforward (see, also, the work in \cite[Proposition~3.1]{HST2020boltzmann}).  
Therefore, for the remainder of the proof, we can assume $|v|\geq 10$.

We use a slightly different decomposition for $\Qs$ than we did above:
\be\label{e.c100401}
	\Qs(f, \vvd^{-q_0})
		= \Big(\int_{B_{2|v|}(v)}
			+ \int_{B_{2|v|}(v)^c}
			\Big)
			 (\vv^{-q_0}  - \vvo^{-q_0}) K_f(v, v') \dd v'
		=: J_1 + J_2,
\ee
The integral over the outer region, $J_2$, can be estimated exactly as we did for $I_3$ in~\eqref{eq:Qs-simple-I3}.  This yields
\be
	|J_2|
		\lesssim \vv^{-q_0-2s} \int_{\R^3} f(v+w) |w|^{\gamma+2s} \dd w.
\ee
After applying \Cref{l:convolution}, we have the desired bound:
\be\label{e.c100402}
	|J_2| \lesssim \vv^{-q_0+\gamma} \KO{f}
		\lesssim \vv^{2+\gamma} \KO{f}.
\ee
The second inequality simply follows from the fact that $2  \geq -q_0$.

The integral over the inner portion, $J_1$, follows nearly the same strategy used for $I_1$ in \eqref{eq:Qs-simple-I1}, except that we do not have the nice bound on the second-order error term $E(v,v')$ used there because $v$ may be near zero in this domain.  Instead, we have only
\be
	|E(v,v')|
		\lesssim |v-v'|^2.
\ee
Using \Cref{l:K-upper-bound-2}, this leads directly to
\be
	|J_1|
		\lesssim \int_{B_{2|v|}(v)} |v-v'|^2 K_f(v,v') \dd v'
		\lesssim \vv^{2-2s} \int_{\R^3}  f(v+w) |w|^{\gamma+2s} \dd w.
\ee
Applying \Cref{l:convolution}, we find
\be\label{e.c100403}
	|J_1|
		\lesssim \vv^{2 + \gamma} \KO{f}.
\ee
The combination of~\eqref{e.c100401}, \eqref{e.c100402}, and~\eqref{e.c100403} concludes the proof.
\end{proof}

\subsection{Proof of \Cref{l:v-alpha}}\label{s.v-alpha}
We now prove our final lemma, which bounds $\Qs(f,\vvd^\alpha)$ for $\alpha \in (0,2s)$.  Its proof is very similar to, but simpler than, that of \Cref{l:Qs-pointwise-base}.
\begin{proof}
First, let us consider the case when $|v| \geq 10$. 
Write
\be
Q_{\rm s}(f,\langle \cdot\rangle^\alpha)(v) = \int_{\R^3} K_f(v,v') [\langle v'\rangle^\alpha - \vv^\alpha] \dd v'
	= I_1 + I_2 + I_3,
\ee
where we use the same decomposition into $I_1$, $I_2$, and $I_3$ as in Lemma \ref{l:Qs-pointwise-base}. 
We first note that
\be
I_2 = \int_{R_0} K_f(v,v') [\langle v'\rangle^\alpha - \vv^\alpha] \dd v' \leq 0,
\ee
since $K_f\geq 0$ and, on $R_0$, $|v'| \leq |v|$. For $I_1$, we again use a Taylor expansion, with the first-order term vanishing because $K_f(v,v+w) = K_f(v,v-w)$, to obtain
\be
	I_1 = \int_{R_{\rm sing}} K_f(v,v') [\langle v'\rangle^\alpha - \vv^\alpha]\dd v'
		\lesssim \vv^{\alpha-2} \int_{R_{\rm sing}}  K_f(v,v') |v-v'|^2 \dd v'.
\ee
Using Lemma \ref{l:K-upper-bound-2}, we then have
\be
	I_1 \lesssim \vv^{\alpha-2} |v|^{2-2s} \int_{\R^3} f(v+w) |w|^{\gamma+2s} \dd w.
\ee
For $I_3$, we use Lemma \ref{l:K-upper-bound-2} to write
\be
	\begin{split}
		I_3
			&= \int_{R_\infty} K_f(v,v') [\langle v'\rangle^\alpha - \vv^\alpha] \dd v'\\
			&\lesssim \int_{\R^3\setminus B_{|v|/4}(v)} K_f(v,v') |v'|^\alpha \dd v'\\
			&= \int_{\R^3\setminus B_{|v|/4}(0)} K_f(v,v+w) |v+w|^\alpha \dd w\\
			&\lesssim \sum_{i=-1}^\infty \int_{B_{2^i |v|}(0) \setminus B_{2^{i-1} |v|}(0)} K_f(v,v+w) |w|^\alpha \dd w\\
			&\lesssim \sum_{i=-1}^\infty  (2^i |v|)^\alpha(2^{i-1} |v|)^{-2s} \int_{\R^3} f(v+w) |w|^{\gamma+2s} \dd w
			\lesssim |v|^{\alpha-2s} \int_{\R^3} f(v+w) |w|^{\gamma+2s} \dd w,
	\end{split}
\ee
since $\alpha < 2s$.  

Combining all estimates, we have
\be
	\Qs(f,\vvd^\alpha)(v)
		\lesssim \vv^{\alpha-2s} \int_{\R^3} f(v+w) |w|^{\gamma+2s} \dd w.
\ee
By \Cref{l:convolution} and the fact that $q_0\geq 2+2s$, we find
\be
	\int_{\R^3} f(v+w) |w|^{\gamma+2s} \dd w
		\lesssim \vv^{(\gamma+2s)_+} \KOO{f}{2+2s}{\pk}
		\leq \vv^{(\gamma+2s)_+} \KO{f},
\ee
as desired. 

If $|v|< 10$, a straightforward modification of the above argument yields the same result. We omit the details of this case because they are simpler than the case $|v|\geq 10$. This concludes the proof.
\end{proof}

\appendix

\section{Convolution estimate}\label{s:conv}

This appendix contains a somewhat elementary estimate on convolutions.  We write it in a nonstandard way for convenience in its application in the proof of \Cref{l:Qs-pointwise-base}.  We provide a proof for completeness.

\begin{lemma}\label{l:convolution}
With $q_0$ as in \Cref{l:Qs-pointwise-base} and $\KO{f}<\infty$, with $\KO{\cdot}$ defined as in \eqref{eq:K0_definition}, there holds
%Fix $\gamma \in (-3,0)$, $s \in (0,1)$, and nonnegative, measurable $f$.  If $\gamma + 2s < 0$, additionally fix $\delta \in (0, 3-|\gamma+2s|)$. 
%Recall the definition of $\KO{\cdot}$ in~\eqref{eq:K0_definition}.
%with $p_\delta$ the H\"older conjugate of
%\be
%	p_\delta' = \frac{\delta - 3}{\gamma+2s}
%		= \frac{3-\delta}{|\gamma+2s|}
%		\qquad\text{ when } \gamma + 2s < 0.
%\ee
%Then
\be
	\begin{split}
	\int_{\R^3} \langle w \rangle^{q_0 - 2 - 2s} f(w) |v-w|^{\gamma+2s} \dd w
		&= \int_{\R^3} \langle v+w\rangle^{q_0-2-2s} f(v+w) |w|^{\gamma+2s} \dd w
		\\&
		\lesssim \vv^{(\gamma+2s)_+} \KO{f},
	\end{split}
\ee
with an implied constant depending only on $q_0$, $p_0$, $\gamma$, $s$, and $c_b$. 
%
%\be\label{e.c080702}
%	\int_{\R^3} \langle w \rangle^{1+\gamma} |v-w|^{\gamma+2s} f(w) \dd w
%		\lesssim \KO{f} \quad\quad \text{ when } -1 \leq \gamma < -2s,
%\ee
%\CH{CHANGE K0s HERE}
%and
%\begin{equation}\label{e.c080701}
%	\int_{\R^3} |v-w|^{\gamma+2s} f(w) \dd w
%		\lesssim \langle v \rangle^{(\gamma + 2s)_+} \KO{f} \quad\quad \text{ otherwise. }
%\end{equation}
%When $\gamma+2s < 0$, the implied constant depends on $\delta$ (recall that $K_0$ has an implicit $\delta$ in its definition).
%
%\CH{$K_0$ SEEMS LIKE BAD NOTATION HERE SINCE $K$ IS THE KERNEL IN OUR OPERATOR}
\end{lemma}
We note that the equality in this statement is obvious, and the true content of \Cref{l:convolution} is the inequality.  We include the equality so that we can use whichever representation of the integral is most convenient.
\begin{proof}
We break our argument up in the moderately soft potentials case ($\gamma + 2s \geq 0$) and the very soft potentials case ($\gamma+2s < 0$).

{\bf Case one: moderately soft potentials.} In this case, we have that 
\be
	|v-w|^{\gamma+2s}
		\lesssim \vv^{\gamma+2s} + \ww^{\gamma+2s}.
\ee
Hence, we immediately see that
\begin{equation}
\begin{split}
	\int_{\R^3} &\ww^{q_0-2-2s} f(w) |v-w|^{\gamma + 2s} \dd w
		\lesssim \int_{\R^3} \ww^{q_0-2-2s}\left(\vv^{\gamma + 2s}  + \ww^{\gamma+2s}\right) f(w) \dd w \\
		&\lesssim \|f\|_{L^1_{q_0 -2 + \gamma}} + \langle v \rangle^{\gamma+2s} \| f \|_{L^1_{q_0 - 2 - 2s}(\R^3)}
	\leq 2\langle v \rangle^{\gamma+2s} \KO{f},
\end{split}
\end{equation}
since $-2s\leq \gamma$. This is precisely the desired inequality in case one.

{\bf Case two: $\gamma+2s < 0$.}
In this case, we require the $L^{\pk}$-norm to handle the singularity at $w\sim v$.  Indeed,
\begin{equation}
\begin{split}
	\int_{\R^3} &\ww^{q_0 - 2 - 2s} f(w) |v-w|^{\gamma+2s} \dd w
		\\&
		= \int_{B_1(v)^c} \ww^{q_0 - 2 - 2s} f(w) |v-w|^{\gamma+2s} \dd w + \int_{B_1(v)} \ww^{q_0 - 2 - 2s} f(w) |v-w|^{\gamma+2s} \dd w
		\\&
		=: J_1 + J_2.
%		& \lesssim \| f \|_{L^1(\R^3)} + \| f \|_{L^{p_{\tiny K}}(\R^3)}
%		= K_0.
\end{split}
\end{equation}
Clearly $J_1$ can be estimated exactly as in case one using the $L^1_{q_0 - 2 +\gamma}$-norm of $f$:
\be
	J_1
		\leq \int_{B_1(v)^c} \ww^{q_0 - 2 - 2s} f(w) \dd w 
		= \vv^{(\gamma+2s)_+} \|f\|_{L^1_{q_0-2+\gamma}}
		\leq \vv^{(\gamma+2s)_+} \KO{f}.
\ee
On the other hand, we use a H\"older's inequality to estimate $J_2$.  Letting $\pk'$ be the dual exponent of $\pk$, we have
\be
	\begin{split}
	J_2
		&\leq \Big( \int_{B_1(v)} |v-w|^{\pk'(\gamma+2s)}\dd w\Big)^\frac{1}{\pk'}
			\Big( \int_{B_1(v)} \ww^{\pk(q_0 - 2-2s)} f(w)^{\pk} \dd w\Big)^\frac{1}{\pk}
		\\&
		\lesssim \KO{f},
	\end{split}
\ee
which is valid because $p_0'(\gamma+2s) > -3$ by the choice of $p_0$ (see \eqref{e.p0def}). 
The proof is complete.
%
%We break our argument up into three cases along the lines of the cases in~\eqref{eq:K0_definition}. \CH{REORDER CASES}
%
%{\bf Case one: $\gamma + 2s \geq 0$.}
%Here we have that $|v-w|^{\gamma+2s} \lesssim \langle v \rangle^{\gamma+2s} + |w|^{\gamma + 2s}$. Then
%\begin{equation}
%\begin{split}
%	\int_{\R^3} \ww^{q_0-2-2s} |v-w|^{\gamma + 2s} f(w) \dd w
%		& \lesssim \int_{\R^3} \ww^{q_0-2-2s}\left(|w|^{\gamma + 2s} f(w) + \langle v \rangle^{\gamma+2s} f(w)\right) \dd w \\
%	&\lesssim \langle v \rangle^{\gamma+2s} \| f \|_{L^1_{\gamma+2s}(\R^3)}
%	\leq \langle v \rangle^{\gamma+2s} K_0.
%\end{split}
%\end{equation}
%This is precisely the desired inequality~\eqref{e.c080701} in case one.
%
%{\bf Case two: $\gamma < -1$ and $\gamma < -2s$.}
%Since $\gamma + 2s$ is negative, we need to use higher integrability of $f$ to control the singular part of the convolution via H\"older's inequality when $|v-w|$ is small.  Indeed:
%\begin{equation}
%\begin{split}
%	\int_{\R^3} |v-w|^{\gamma+2s} f(w) \dd w
%		& = \int_{|v-w| > 1} |v-w|^{\gamma+2s} f(w) \dd w + \int_{|v-w| \leq 1} |v-w|^{\gamma+2s} f(w) \dd w \\
%		& \lesssim \| f \|_{L^1(\R^3)} + \| f \|_{L^{p_{\tiny K}}(\R^3)}
%		= K_0.
%\end{split}
%\end{equation}
%This is precisely the desired inequality~\eqref{e.c080701} in case two. 
%
%{\bf Case three: $-1 \leq \gamma < -2s$.}
%This is handled analogously to Case 2, except that the weight function $\langle w \rangle^{1+\gamma}$ gets ``inherited'' by the
%Lebesgue norms on $f$.  The proof is omitted.
\end{proof}

\bibliographystyle{abbrv}
\bibliography{decay}

\end{document}